\documentclass[11pt,a4paper,reqno]{article}

\usepackage{amssymb}
\usepackage{latexsym}
\usepackage{amsmath}
\usepackage{graphicx}
\usepackage{amsthm}
\usepackage{empheq}
\usepackage{bm}
\usepackage{booktabs}
\usepackage[dvipsnames]{xcolor}
\usepackage{pagecolor}
\usepackage{subcaption}
\usepackage{tikz}
\usepackage{cite}

\usepackage[margin=2.8cm]{geometry}

\numberwithin{equation}{section}

\theoremstyle{definition}
\newtheorem{theorem}{Theorem}[section]
\newtheorem{corollary}[theorem]{Corollary}
\newtheorem{proposition}[theorem]{Proposition}
\newtheorem{definition}[theorem]{Definition}
\newtheorem{example}[theorem]{Example}
\newtheorem{conjecture}[theorem]{Conjecture}
\newtheorem{notation}[theorem]{Notation}
\newtheorem{remark}[theorem]{Remark}
\newtheorem{lemma}[theorem]{Lemma}

\newtheorem{problem}[theorem]{Problem}

\newcommand\qbin[3]{\left[\begin{matrix} #1 \\ #2 \end{matrix}\right]_{#3}}
\newcommand\bbq[1]{\bm{b}_q(#1)}
\newcommand{\numberset}{\mathbb}
\newcommand{\N}{\numberset{N}}
\newcommand{\Z}{\numberset{Z}}

\newcommand{\F}{\numberset{F}}

\newcommand{\inv}{\textnormal{inv}}

\newcommand{\mC}{\mathcal{C}}

\newcommand{\mF}{\mathcal{F}}

\newcommand{\rk}{\textnormal{rk}}

\newcommand{\mP}{\mathcal{P}}

\newcommand{\mat}{\F_q^{n \times m}}

\newcommand{\supp}{\textnormal{supp}}

\newtheoremstyle{claim}
  {\topsep}
  {\topsep}
  {}
  {}
  {\itshape}
  {}
  {.5em}
  {\underline{\thmname{#1}}\thmnumber{ #2}.\thmnote{ (#3)}}
\theoremstyle{claim}
\newtheorem{claim}{Claim}

\newcommand*{\myproofname}{Proof of the claim}
\newenvironment{clproof}[1][\myproofname]{\begin{proof}[#1]}{\end{proof}}

\usepackage{titling}
\usepackage{hyperref}

\usepackage{authblk}
\title{\textbf{Rook Theory of the Etzion-Silberstein Conjecture}}

\author{Anina Gruica\thanks{A. G. is supported by the Dutch Research Council through grant OCENW.KLEIN.539.} \ and Alberto Ravagnani\thanks{A. R. is supported by the Dutch Research Council through grants VI.Vidi.203.045, 
OCENW.KLEIN.539, 
and by the Royal Academy of Arts and Sciences of the Netherlands.}}
\date{}                    
\affil{Department of Mathematics and Computer Science \\ Eindhoven University of 
	Technology, the Netherlands}

\usepackage{setspace}
\begin{document}








	\maketitle
	
	\thispagestyle{empty}
	
	\begin{abstract}
		In 2009, Etzion and  Siberstein proposed a conjecture on the largest dimension  of a linear space of matrices over a finite field in which all nonzero matrices are supported
		on a Ferrers diagram and have 
		rank bounded below by a given integer. Although several cases of the conjecture have been established in the past decade, proving or disproving it remains to date a wide open problem. In this paper, we take a new look at the Etzion-Siberstein Conjecture,
		investigating its connection with  rook theory.
		Our results show that the combinatorics behind this open problem is closely linked 
		to the theory of $q$-rook polynomials associated with Ferrers diagrams, as defined by Garsia and  Remmel. 
		In passing, we give a closed formula for the trailing degree of the $q$-rook polynomial associated with a Ferrers diagram in terms of the cardinalities of its diagonals.
		The combinatorial approach taken in this paper 
		 allows us to 
		establish some new instances of the Etzion-Silberstein Conjecture using a non-constructive argument. We also 
		solve the  
		asymptotic version of the conjecture over large finite fields, answering a current open question. 
	\end{abstract}

	\bigskip

\section*{Introduction}

Linear spaces of matrices whose ranks obey various types of  constraints have been extensively investigated in algebra and combinatorics with many approaches and techniques; see~\cite{delsarte1978bilinear,meshulam1985maximal,gelbord2002spaces,seguins2015classification,eisenbud1988vector,lovasz1989singular,draisma2006small,dumas2010subspaces} and the references therein, among many others.
In~\cite{etzion2009error}, Etzion and Silberstein consider
linear spaces of matrices over a finite field $\F_q$ that are supported on a Ferrers diagram~$\mF$ and in which every nonzero matrix has rank bounded below by a certain integer~$d$.
For the application considered in~\cite{etzion2009error}, it is particularly relevant to determine the largest dimension of a linear space having the described properties, which we call an $[\mF,d]_q$-space in the sequel.

In the same paper, Etzion and Silberstein 
prove a bound on the dimension of
any $[\mF,d]_q$-space, which is computed by deleting $d-1$ lines (rows or columns) of the diagram $\mF$ and determining the smallest area that can be obtained in this way; see Theorem~\ref{thm:dimbound} below for a precise statement.
They also conjecture that said bound is sharp for any pair $(\mF,d)$ and any field size $q$, a problem that goes under the name of the \textit{Etzion-Silberstein Conjecture}.

Since 2009, several cases of the conjecture have been settled using various approaches, but proving or disproving it remains to date an open problem.
Most instances of the conjecture that have been proved so far rely on ``case-by-case'' studies,
which divide Ferrers diagrams into classes and design proof techniques that work for a specific class.
The natural consequence of this is the lack of a ``unified'' approach to solving the conjecture, which in turn makes it difficult to understand the potentially very rich combinatorial theory behind~it. One of the goals of this paper is to fill in  this gap.

In~\cite{antrobus2019maximal}, 
Antrobus and Gluesing-Luerssen propose a new research direction and initiate the study of the Etzion-Silberstein Conjecture in the asymptotic regime. More precisely, they investigate 
for which pairs $(\mF,d)$ a randomly chosen space meets the Etzion-Silberstein Bound with high probability over a sufficiently large finite field. 
In the same article, they also answer the question for a class of pairs~$(\mF,d)$ called \textit{MDS-constructible}.  
The question asked by Antrobus and Gluesing-Luerssen generalizes the problem of determining whether or not MRD codes in the rank metric are sparse for large field sizes.

The goal of this paper is to explore the combinatorics behind the Etzion-Silberstein Conjecture, with a particular focus on rook theory and the theory of 
Catalan numbers. 
The approach taken in this paper will also allow us to establish 
the conjecture for some parameter sets using a non-constructive approach, and to answer an open question from~\cite{antrobus2019maximal}.
More in detail, the contribution made by this paper is threefold.

\begin{enumerate}
    \item We study the combinatorics of MDS-constructible pairs, as defined in~\cite{antrobus2019maximal}, showing that a pair $(\mF,d)$ is MDS-constructible precisely when the Etzion-Silberstein Bound of~\cite{etzion2009error} coincides with the trailing degree of the $(d-1)$th $q$-rook polynomial associated with the Ferrers diagram $\mF$. This gives a curious, purely combinatorial characterization of MDS-constructible pairs, which 
    we prove by giving a 
    closed formula for the trailing degree 
    of the $q$-rook polynomial in terms of the diagonals of the underlying Ferrers diagram.
    The latter result does not appear to be combinatorially obvious.
    
    \item We solve the asymptotic analogue of the Etzion-Silberstein Conjecture, determining for which dimensions
    $k$ and for which pairs $(\mF,d)$ the $k$-dimensional $[\mF,d]_q$-spaces are sparse or dense as the field size goes to infinity. This completes the results obtained in~\cite{antrobus2019maximal} 
    by answering an open question from the same paper using a combinatorial approach based on a classical result by Haglund. The idea behind our proof also
    suggests a non-constructive approach to the Etzion-Silberstein Conjecture, which we use to establish it in some new cases.

    \item The theory of MDS-constructible pairs appears to be closely related to that of Catalan numbers. In this paper, we show that these count the MDS-constructible pairs of the form $(\mF,2)$. We also obtain formulas for the MDS-constructible pairs of the form $(\mF,3)$ for when $\mF$ is a square Ferrers diagram.
\end{enumerate}

This paper is organized as follows.
Section~\ref{sec:1} states the Etzion-Silberstein Conjecture and introduces the needed preliminaries. The combinatorics of MDS-constructible pairs and their connection with $q$-rook polynomials is investigated in Section~\ref{sec:2}. We solve the asymptotic version of the Etzion-Silberstein Conjecture in Section~\ref{sec:3} and present the new cases we establish in Section~\ref{sec:4}. Closed formulas for the number of some MDS-constructible pairs are given in Section~\ref{sec:5}, where we also highlight their link with Catalan numbers.

\section{The Etzion-Silberstein Conjecture}
\label{sec:1}

Throughout this paper, $q$ denotes a prime power and $\F_q$ is the finite field with $q$ elements. We let $m$ and $n$ denote positive
integers and $\smash{\mat}$ the space of 
$n \times m$ matrices with entries in~$\F_q$. For an integer $i \in \N$, we let $[i]
=\{1,\dots,i\}$. We start by defining Ferrers diagrams.

\begin{definition}
An $n \times m$ \textbf{Ferrers diagram} is a subset $\mF \subseteq [n] \times [m]$ with the following properties:
\begin{enumerate}
    \item $(1,1) \in \mF$ and $(n,m) \in \mF$;
    \item if $(i,j) \in \mF$ and $j < m$, then $(i,j+1) \in \mF$ \ (right-aligned);
    \item if $(i,j) \in \mF$ and $i >1$, then $(i-1,j) \in \mF$ \ (top-aligned).
\end{enumerate}
We often denote a Ferrers diagram $\mF$ as an array $[c_1, \dots, c_m]$ of positive integers, where for all $1 \le j \le m$ we set $$c_j=|\{(i,j) : 1 \le i \le n, \, (i,j) \in \mF \}|.$$ 
By the definition of Ferrers diagram, we have $1 \le c_1 \le c_2 \le \dots \le c_m=m$. For $1 \le i \le n$, the $i$th \textbf{row} of $\mF$ is the set of $(i,j) \in \mF$ with $j \in [m]$. Analogously, for $1 \le j \le m$, the $j$th \textbf{column} of $\mF$ is the set of $(i,j) \in \mF$ with $i \in [n]$.
\end{definition}

Ferrers diagrams are often represented as 2-dimensional arrays of ``dots'', as 
Figure~\ref{F-F133466} illustrates.

     \begin{figure}[ht]
    \centering
     {\small
     \begin{tikzpicture}[scale=0.35]
         \draw (5.5,2.5) node (b1) [label=center:$\bullet$] {};
         \draw (5.5,3.5) node (b1) [label=center:$\bullet$] {};
         \draw (5.5,4.5) node (b1) [label=center:$\bullet$] {};
         \draw (5.5,5.5) node (b1) [label=center:$\bullet$] {};
         \draw (5.5,6.5) node (b1) [label=center:$\bullet$] {};
         
         \draw (4.5,2.5) node (b1) [label=center:$\bullet$] {};
         \draw (4.5,3.5) node (b1) [label=center:$\bullet$] {};
         \draw (4.5,4.5) node (b1) [label=center:$\bullet$] {};
         \draw (4.5,5.5) node (b1) [label=center:$\bullet$] {};
         \draw (4.5,6.5) node (b1) [label=center:$\bullet$] {};
       \

         \draw (3.5,3.5) node (b1) [label=center:$\bullet$] {};
         \draw (3.5,4.5) node (b1) [label=center:$\bullet$] {};
         \draw (3.5,5.5) node (b1) [label=center:$\bullet$] {};
         \draw (3.5,6.5) node (b1) [label=center:$\bullet$] {};

         \draw (2.5,4.5) node (b1) [label=center:$\bullet$] {};
         \draw (2.5,5.5) node (b1) [label=center:$\bullet$] {};
         \draw (2.5,6.5) node (b1) [label=center:$\bullet$] {};
         
         \draw (1.5,4.5) node (b1) [label=center:$\bullet$] {};
         \draw (1.5,5.5) node (b1) [label=center:$\bullet$] {};
         \draw (1.5,6.5) node (b1) [label=center:$\bullet$] {};

       \draw (0.5,6.5) node (b1) [label=center:$\bullet$] {};
     \end{tikzpicture}
     }
     \caption{The Ferrers diagram $\mF=[1,3,3,4,5,5]$.}
     \label{F-F133466}
     \end{figure}
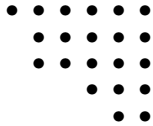

In this paper, we are interested in linear spaces made of matrices that are \textit{supported} on a Ferrers diagram, in the following precise sense.

\begin{definition} \label{defsupp}
The \textbf{support} of a matrix $M \in \mat$ is the index set of its nonzero entries, i.e.,  $\supp(M)=\{(i,j) \mid M_{ij} \neq 0\} \subseteq [n] \times [m]$.
If $\mF$ is an $n \times m$ Ferrers diagram, then we say that
$M \in \mat$ is \textbf{supported} on~$\mF$ if $\supp(M) \subseteq \mF$. We denote by
$\F_q[\mF]$ the $\F_q$-linear space of matrices that are supported on the Ferrers diagram $\mF$.
\end{definition}

Following the notation of Definition~\ref{defsupp}, $\F_q[\mF]$ has dimension $|\mF|$ over $\F_q$.
We study linear spaces of matrices in which all nonzero matrices have rank bounded from below by a given integer and are supported on a Ferrers diagram.

\begin{definition}
Let $\mF$ be an $n \times m$ Ferrers diagram and let $d \ge 1$ be an integer. 
An \textbf{$[\mF,d]_q$-space}
is an $\F_q$-linear subspace 
$\mC \le \F_q[\mF]$
with the property that $\rk(M) \ge d$ for all nonzero
matrices $M \in \mC$. 
\end{definition}

In coding theory, $[\mF,d]_q$-spaces naturally arise
in the construction of large \textit{subspace codes} via the so-called \textit{multilevel construction}; see~\cite{etzion2009error} for the details.
In~\cite[Theorem 1]{etzion2009error}, Etzion and  Silbertein establish an upper bound 
for the dimension of an $[\mF,d]_q$-space.
In order to state the bound, we need to introduce 
the following quantities.

\begin{notation} \label{not:kappa}
Let $\mF=[c_1,\dots, c_m]$ be an $n \times m$ Ferrers diagram and let $1 \le d \le \min\{n,m\}$ be an integer. For $0 \le j \le d-1$, let
$\kappa_j(\mF,d)=\sum_{t=1}^{m-d+1+j} \max\{c_t-j,0\}$. We then set
\begin{align} \label{min_b}
    \kappa(\mF,d)= \min\{\kappa_j(\mF,d) \mid 0 \le j \le d-1\}.
\end{align}
\end{notation}

Note that, by definition,
$\kappa_j(\mF,d)$ is the number of points in the Ferrers diagram $\mF$ after removing the topmost $j$ rows and the rightmost $d-1-j$ columns. 
We can now state the bound proved by Etzion and Silberstein.

\begin{theorem}[see \textnormal{\cite[Theorem 1]{etzion2009error}}] \label{thm:dimbound}
Let $\mF$ be an $n \times m$ Ferrers diagram and let $1 \le d \le \min\{n,m\}$ be an integer. Let $\mC$ be an $[\mF,d]_q$-space.
We have \begin{align*}
    \dim(\mC) \le \kappa(\mF,d).
\end{align*}
\end{theorem}

We call an
$[\mF,d]_q$-space that
meets the bound of Theorem~\ref{thm:dimbound} with equality \textbf{optimal}.
When $\mF=[n] \times [m]$,
the bound reads as $\dim(\mC)\le \max\{n,m\}(\min\{n,m\}-d+1)$, which is the well-known Singleton-type bound for a rank-metric code established by Delsarte; see~\cite[Theorem 5.4]{delsarte1978bilinear}. Subspaces of $\mat$ meeting the Singleton-type bound with equality are called
\textit{maximum-rank-distance codes} (\textit{MRD codes} in short)
and form a central theme in contemporary coding theory and combinatorics; see \cite{koetter2008coding,gabidulin,SKK,roth1991maximum,delsarte1978bilinear,sheekey2020new,braun2016existence,lewis2020rook,gorla2018rankq,schmidt2020quadratic,csajbok2017maximum} 
among many others.

\begin{example}
Let $\mF=[1,3,3,4,5,5]$ be the Ferrers diagram of Figure~\ref{F-F133466}. Then an
$[\mF,4]_q$-space 
is optimal if its dimension is $7$, where the minimum in~\eqref{min_b} can be attained by deleting the top row and the~2 rightmost columns.
\end{example}

In~\cite{etzion2009error}, 
Etzion and Silberstein conjecture that the bound of Theorem~\ref{thm:dimbound} is sharp for all pairs $(\mF,d)$ and for any field size $q$; see~\cite[Conjecture~1]{etzion2009error}.
The conjecture has been proven in several cases; see for instance~\cite{etzion2009error, etzion2016optimal, gorla2017subspace, silberstein2013new, silberstein2015subspace, trautmann2011new, zhang2019constructions, liu2019constructions, ballico2015linear}.
At the time of writing this paper, it is not known whether or not optimal $[\mF,d]_q$-spaces 
exist for all parameters, i.e., whether the conjecture by Etzion-Silberstein holds.

\begin{conjecture}[Etzion-Silberstein~\cite{etzion2009error}] \label{conj:ES}
For every prime power $q$, every $n \times m$ Ferrers diagram $\mF$, and every integer $1 \le d \le \min\{n,m\}$, there exists an 
$[\mF,d]_q$-space of maximum dimension $\kappa(\mF,d)$.
\end{conjecture}

Note that Conjecture~\ref{conj:ES} is stated for finite fields only and it is false in general for infinite fields; see~\cite{gorla2017subspace}.

This paper studies some combinatorial problems that are naturally connected with
Conjecture~\ref{conj:ES}, with particular focus on rook theory. In passing, we will show how some instances of the conjecture can be established using a non-constructive approach; see Section~\ref{sec:existence}.


\section{Combinatorics of MDS-Constructible Pairs} \label{sec:2}

There exists a special class of pairs 
$(\mF,d)$ for which the bound of Theorem~\ref{thm:dimbound} can be attained with equality, for $q$ sufficiently large, using \textit{MDS error-correcting codes}; see~\cite{macwilliams1977theory} for the coding theory terminology.
In~\cite{antrobus2019maximal},
these pairs are called \textit{MDS-constructible} for natural reasons.
The construction of 
$[\mF,d]_q$-spaces based on MDS codes can be found in~\cite{gorla2017subspace,etzion2016optimal}, although it dates back to~\cite{roth1991maximum}, where it appears in a slightly different context. 
In order to state the existence result corresponding to this construction, we need the following concept.

\begin{notation} \label{not:diag}
For $1 \le r \le m+n-1$, define the 
\textbf{$r$th diagonal} of the $n \times m$ matrix board as
$$D_r = \{(i,j) \in [n] \times [m] : j-i = m-r\}
\subseteq [n] \times [m].$$
\end{notation}

Note that in Notation~\ref{not:diag} we consider more diagonals than in~\cite[Definition VI.5]{antrobus2019maximal}. This choice will play a crucial role in some of our results. 
We are interested in the number of elements on the diagonals of a Ferrers diagram.

\begin{example}
The elements on the diagonals of $\mF=[1,3,3,4,6,6,6]$ can be seen in Figure~\ref{fig:diag}. We have $|D_i \cap \mF| = i$ for $1 \le i \le 6$, $|D_7 \cap \mF| = 6$, $|D_8 \cap \mF| = 2$, and $|D_i \cap \mF|= 0$ for $9 \le i \le 12$. 
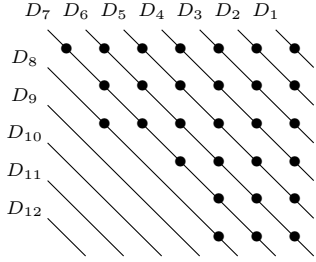
\begin{figure}[ht]
\centering
{\small
     \begin{tikzpicture}[scale=0.5]
    \draw (0,7) to (6,1);
    \draw (1,7) to (7,1);
    \draw (2,7) to (7,2);
    \draw (3,7) to (7,3);
    \draw (4,7) to (7,4);
    \draw (5,7) to (7,5);
    \draw (6,7) to (7,6);
    \draw (0,6) to (5,1);
    \draw (0,5) to (4,1);
    \draw (0,4) to (3,1);
    \draw (0,3) to (2,1);
    \draw (0,2) to (1,1);
    
    \draw (5.75,7.5) node (b1) [label=center:${\scriptstyle D_1}$] {};
    \draw (4.75,7.5) node (b1) [label=center:${\scriptstyle D_2}$] {};
    \draw (3.75,7.5) node (b1) [label=center:${\scriptstyle D_3}$] {};
    \draw (2.75,7.5) node (b1) [label=center:${\scriptstyle D_4}$] {};
    \draw (1.75,7.5) node (b1) [label=center:${\scriptstyle D_5}$] {};
    \draw (0.75,7.5) node (b1) [label=center:${\scriptstyle D_6}$] {};
    \draw (-0.25,7.5) node (b1)
    [label=center:${\scriptstyle D_7}$] {};
    
    \draw (-0.6,6.25) node (b1) [label=center:${\scriptstyle D_8}$] {};
    \draw (-0.6,5.25) node (b1) [label=center:${\scriptstyle D_9}$] {};
    \draw (-0.6,4.25) node (b1) [label=center:${\scriptstyle D_{10}}$] {};
    \draw (-0.6,3.25) node (b1) [label=center:${\scriptstyle D_{11}}$] {};
    \draw (-0.6,2.25) node (b1) [label=center:${\scriptstyle D_{12}}$] {};
    
         \draw (6.5,1.5) node (b1) [label=center:$\bullet$] {};
         \draw (6.5,2.5) node (b1) [label=center:$\bullet$] {};
         \draw (6.5,3.5) node (b1) [label=center:$\bullet$] {};
         \draw (6.5,4.5) node (b1) [label=center:$\bullet$] {};
         \draw (6.5,5.5) node (b1) [label=center:$\bullet$] {};
         \draw (6.5,6.5) node (b1) [label=center:$\bullet$] {};
    
         \draw (5.5,1.5) node (b1) [label=center:$\bullet$] {};
         \draw (5.5,2.5) node (b1) [label=center:$\bullet$] {};
         \draw (5.5,3.5) node (b1) [label=center:$\bullet$] {};
         \draw (5.5,4.5) node (b1) [label=center:$\bullet$] {};
         \draw (5.5,5.5) node (b1) [label=center:$\bullet$] {};
         \draw (5.5,6.5) node (b1) [label=center:$\bullet$] {};
         
         \draw (4.5,1.5) node (b1) [label=center:$\bullet$] {};
         \draw (4.5,2.5) node (b1) [label=center:$\bullet$] {};
         \draw (4.5,3.5) node (b1) [label=center:$\bullet$] {};
         \draw (4.5,4.5) node (b1) [label=center:$\bullet$] {};
         \draw (4.5,5.5) node (b1) [label=center:$\bullet$] {};
         \draw (4.5,6.5) node (b1) [label=center:$\bullet$] {};
       \

         \draw (3.5,3.5) node (b1) [label=center:$\bullet$] {};
         \draw (3.5,4.5) node (b1) [label=center:$\bullet$] {};
         \draw (3.5,5.5) node (b1) [label=center:$\bullet$] {};
         \draw (3.5,6.5) node (b1) [label=center:$\bullet$] {};

         \draw (2.5,4.5) node (b1) [label=center:$\bullet$] {};
         \draw (2.5,5.5) node (b1) [label=center:$\bullet$] {};
         \draw (2.5,6.5) node (b1) [label=center:$\bullet$] {};
         
          \draw (1.5,4.5) node (b1) [label=center:$\bullet$] {};
         \draw (1.5,5.5) node (b1) [label=center:$\bullet$] {};
         \draw (1.5,6.5) node (b1) [label=center:$\bullet$] {};

       \draw (0.5,6.5) node (b1) [label=center:$\bullet$] {};
     \end{tikzpicture}
     }
\caption{Graphical representation of the diagonals and of the Ferrers diagram $\mF=[1,3,3,4,6,6]$ in the $6 \times 7$ matrix board.}
\label{fig:diag}
\end{figure}
\end{example}

The construction of $[\mF,d]_q$-spaces based on MDS error-correcting codes gives the following lower bound on their dimension; the case of algebraically closed fields is treated in~\cite[Section VI]{antrobus2019maximal}.

\begin{theorem}[see \cite{roth1991maximum,gorla2017subspace,etzion2016optimal}] \label{construc}
Let $\mF$ be an $n \times m$ Ferrers diagram with $m \ge n$ and let $1\le d \le n$ be an integer. If 
$\smash{q \ge \max\{|D_i \cap \mF| \, : \, 1 \le i \le m\}-1}$, then there exists an $[\mF,d]_q$-space of dimension $\smash{\sum_{i=1}^m \max\{0, |D_i \cap \mF|-d+1\}}$.
\end{theorem}

A pair $(\mF,d)$ MDS-constructible if the bound of Theorem~\ref{thm:dimbound} is attained with equality, for $q$ large,
via the construction of Theorem~\ref{construc}.

\begin{definition} \label{def:mdsconstr}
Let $\mF$ be an $n \times m$ Ferrers diagram with $m \ge n$ and let $1\le d \le n$ be an integer.
The pair $(\mF,d)$ is \textbf{MDS-constructible} if
\begin{equation} \label{eq-MDSc}
\kappa(\mF,d) = \sum_{i=1}^{m}\max\{0, |D_i \cap \mF|-d+1\}.
\end{equation}
\end{definition}

\begin{remark} \label{rem:constr}
We briefly illustrate the construction used in the proof of Theorem~\ref{construc}, where we follow~\cite[Section IV]{etzion2016optimal} or similarly~\cite[Theorem 32]{gorla2017subspace}.
Let $\mF$ be an $n \times m$ Ferrers diagram with $m \ge n$ and let $1 \le d \le n$ be an integer such that the pair $(\mF,d)$ is MDS-constructible. Let $q \ge \max\{|D_i \cap \mF| : 1 \le i \le m\}-1$. Consider the set $I=\{1 \le i \le m : |D_i \cap \mF| \ge d\}=\{i_1,\dots,i_{\ell}\}$ and for all $i \in I$ let $n_i=|D_i \cap \mF|$. By our assumption on $q$, there exists a linear MDS code $C_i \le \F_q^{n_i}$ of minimum distance $d$. Now for $(x_{i_1}, \dots, x_{i_\ell}) \in C_{i_1} \times \dots \times C_{i_{\ell}}$ we let $M=M(x_{i_1}, \dots, x_{i_\ell}) \in \F_q[\mF]$ be the matrix with the vector $x_{i_j}$ in the positions indexed by $D_{i_j} \cap \mF$ for all $1\le j \le \ell$, and with zeros everywhere else. Let
\begin{align*}
    \mC=\{M(x_{i_1}, \dots, x_{i_\ell}) \, : \, (x_{i_1}, \dots, x_{i_\ell}) \in C_{i_1} \times \dots \times C_{i_{\ell}} \}.
\end{align*}
One can show that $\mC$ is an optimal $[\mF,d]_q$-space of dimension $\sum_{j=1}^{\ell}(n_{i_j}-d+1)$, which in turn establishes Theorem~\ref{construc}.
\end{remark}

Before diving into the rook theory of MDS-constructible pairs, we include a few observations about Definition~\ref{def:mdsconstr} and in particular on the restriction $m \ge n$. 

\begin{remark} \label{rmk:symm} 
\begin{enumerate}
    \item The sum on the RHS of~\eqref{eq-MDSc}
is not symmetric in $n$ and $m$, even though the assumption $m \ge n$ is not restrictive (up a transposition of the Ferrers diagram, if necessary). In particular,
which value between $n$ and $m$ is the largest plays a role, \textit{a priori}, in the definition of an MDS-constructible pair. At the end of this section we will return to this point and show that MDS-constructible pairs admit a characterization that is perfectly symmetric in $n$ and $m$ and that 
has a specific rook theory significance;
see Theorems~\ref{th:trai} and~\ref{prop:newmdsconstr}
below. For that characterization, it is crucial to consider all the $m+n-1$ diagonals
introduced in Notation~\ref{not:diag} (and not only the first~$m$).

\item Definition~\ref{def:mdsconstr} does not reflect ``optimality'' when $d=1$. Indeed, when $d=1$ we have $\kappa(\mF,1)=|\mF|$ for any $n \times m$ Ferrers diagram. In particular, the bound of Theorem~\ref{thm:dimbound} is sharp and attained by the ambient space $\F_q[\mF]$, which often makes the construction described in Remark~\ref{rem:constr} suboptimal.
The definition of MDS-constructible pair that we will propose at the end of this section solves this very minor inconvenience. 
\end{enumerate}
\end{remark}

A natural question is whether MDS-constructible pairs $(\mF,d)$ admit a purely combinatorial characterization in terms of known
structural invariants of a Ferrers diagram. In this section, we will answer the question in the affirmative, proving that MDS-constructible pairs are precisely those for which the 
Etzion-Silberstein Bound of Theorem~\ref{thm:dimbound} takes the same value as the trailing degree of the $(d-1)$th $q$-rook polynomial associated with $\mF$; see Corollary~\ref{cor:main}.
This curious fact does not appear to be have an obvious combinatorial explanation.

The main tool in our approach 
is a closed formula for the trailing degree of a $q$-rook polynomial in terms of the diagonals of the underlying Ferrers diagram; see Theorem~\ref{th:trai}. We start by recalling the needed rook theory terminology.

\begin{definition}
An $n \times m$ \textbf{non-attacking rook placement} is a subset $C \subseteq [n] \times [m]$ with the property that
no two elements of $C$ share the same row or column index. In this context, the elements of $C$ are called \textbf{rooks}.
\end{definition}

In~\cite{GaRe86}, Garsia and  Remmel propose a definition for the $q$-analogue of the rook polynomial associated with a Ferrers diagram. The definition is based on the following quantity.

\begin{notation} \label{not:invrook}
Let $\mF$ be an $n \times m$ Ferrers diagram
and let $C \subseteq \mF$ be an $n \times m$ non-attacking rook placement. We denote by $\inv(C,\mF)$ the number computed as follows: Cross out all the dots from~$\mF$  that either correspond to a rook of $C$,
or are
above or to the right of any rook of $C$; then $\inv(C,\mF)$ is the number of dots
of $\mF$ not crossed out. 
\end{notation}

The $q$-rook polynomials of a Ferrers diagram are defined as follows.

\begin{definition} \label{def_qpoly}
Let $\mF$ be an $n \times m$ Ferrers diagram and let $r \ge 0$ be an integer. The $r$th $q$-rook polynomial of $\mF$ is
$$R_q(\mF,r)= \sum_{C \in \textnormal{NAR}(\mF,r)} q^{\inv(C,\mF)} \, \in \Z[q],$$
where $\textnormal{NAR}(\mF,r)$ denotes the set of $n \times m$ non-attacking rook placements $C \subseteq \mF$ having cardinality~$|C|=r$.
\end{definition}

We also recall that the \textbf{trailing degree} of a polynomial $R=\sum_{i} a_iq^i \in \Z[q]$ is the minimum $i$ with $a_i \neq 0$, where the zero polynomial has trailing degree $-\infty$.
Therefore, following the notation of Definition~\ref{def_qpoly}, the trailing degree of the $r$th $q$-rook polynomial
of~$\mF$ is the minimum value of $\inv(C,\mF)$, as~$C$ ranges over the set
$\textnormal{NAR}(\mF,r)$, whenever the 
$r$th $q$-rook polynomial is nonzero.
Since the trailing degree of the $q$-rook polynomial will play a crucial role in this paper, we introduce a symbol for it.

\begin{notation}
Following the notation of Definition~\ref{def_qpoly}, we denote the trailing degree of the polynomial $R_q(\mF,r)$ by $\tau(\mF,r)$.
\end{notation}

We illustrate the concepts introduced before with an example.

\begin{example}
Consider the $5 \times 5$ Ferrers diagram $\mF=[1,3,3,4,5]$. Figure~\ref{F-F13335} represents a non-attacking rook placement $C \in \textnormal{NAR}(\mF,3)$, where we also illustrate the deletions that compute $\inv(C,\mF)$ according to Notation~\ref{not:invrook}. Note that we have $\inv(C,\mF) = 5$. Moreover, the third $q$-rook polynomial of $\mF$ can be computed as follows:
\begin{align*}
  R_q(\mF,3)= \sum_{C \in \textnormal{NAR}(\mF,3)} q^{\inv(C,\mF)} = 6q^3+ 18q^4 + 27q^5  + 28q^6 + 20q^7 + 11q^8  + 4q^9 + q^{10}.
\end{align*}
Therefore, $\tau(\mF,3)=3$.
     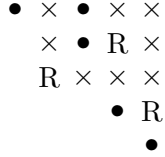
\begin{figure}[ht]
    \centering
     {
     \begin{tikzpicture}[scale=0.45]
         \draw (4.5,2.5) node (b1) [label=center:$\bullet$] {};
         \draw (4.5,3.5) node (b1) [label=center:R] {};
         \draw (4.5,4.5) node (b1) [label=center:$\times$] {};
         \draw (4.5,5.5) node (b1) [label=center:$\times$] {};
         \draw (4.5,6.5) node (b1) [label=center:$\times$] {};
       \
        
         \draw (3.5,3.5) node (b1) [label=center:$\bullet$] {};
         \draw (3.5,4.5) node (b1) [label=center:$\times$] {};
         \draw (3.5,5.5) node (b1) [label=center:R] {};
         \draw (3.5,6.5) node (b1) [label=center:$\times$] {};

         \draw (2.5,4.5) node (b1) [label=center:$\times$] {};
         \draw (2.5,5.5) node (b1) [label=center:$\bullet$] {};
         \draw (2.5,6.5) node (b1) [label=center:$\bullet$] {};
         
         \draw (1.5,4.5) node (b1) [label=center:R] {};
         \draw (1.5,5.5) node (b1) [label=center:$\times$] {};
         \draw (1.5,6.5) node (b1) [label=center:$\times$] {};

       \draw (0.5,6.5) node (b1) [label=center:$\bullet$] {};
     \end{tikzpicture}
     }
     \caption{The non-attacking rook placement $C=\{(2,4), (3,2), (4,5)\}$. The rooks are marked with ``R''. The symbol ``$\times$'' illustrates the cancellations operated to compute $\inv(C,\mF)$.}
     \label{F-F13335}
     \end{figure}
\end{example}

In~\cite[Theorem 1]{haglund}, 
Haglund shows that the theory of $q$-rook polynomials for Ferrers diagrams is closely connected with the problem of enumerating the number of matrices having prescribed rank and $\mF$ as profile.

\begin{notation} \label{notPq}
Let $\mF$ be an $n \times m$ Ferrers diagram and let $r \ge 0$ be an integer. We denote by $P_q(\mF,r)$ the size of the set of matrices
$M \in \F_q[\mF]$ of rank $r$. 
\end{notation}

The next result was established in~\cite{gluesing2020partitions} and it heavily relies on~\cite[Theorem 1]{haglund}.

\begin{theorem}[see~\textnormal{\cite[Proposition 7.11]{gluesing2020partitions}}]
\label{degHag}
Let $\mF$ be an $n \times m$ Ferrers diagram and let $r \ge 0$ be an integer.
Then $P_q(\mF,r)$ is a polynomial in $q$ whose degree satisfies
$$\deg(P_q(\mF,r)) + \tau(\mF,r)= |\mF|.$$
\end{theorem}

In some of our statements, we will assume 
$\kappa(\mF,r) \ge 1$. The next result shows that this assumption only excludes pairs $(\mF,d)$ 
for which the corresponding $q$-rook polynomial is the zero polynomial, and 
for which
Conjecture~\ref{conj:ES} is trivial.

\begin{proposition} \label{prop:exist_r}
Let $\mF$ be an $n \times m$ Ferrers diagram and let $1 \le r \le \min\{n,m\}$ be an integer. Then 
$\kappa(\mF,r) \ge 1$ if and only if 
there exists a matrix 
$M \in \F_q[\mF]$ with $\rk(M) \ge r$.
\end{proposition}

\begin{proof}
Note that $\kappa(\mF,r) \ge 1$ implies that $|D_i \cap \mF|=i$ for all $1 \le i \le r$, as we will show later in Claim~\ref{cl:1}. Therefore, under this assumption, there exists a matrix $M \in \F_q[\mF]$ with $\rk(M)\ge r$, for example by choosing $M$ to be the matrix with 1's in the positions indexed by~$D_r$ and 0's elsewhere. For the other implication, suppose $\kappa(\mF,r) =0$. From Theorem~\ref{thm:dimbound} it follows that there cannot exist $M \in \F_q[\mF]$ with $\rk(M) \ge r$, since we assume that $r \ge 1$.
\end{proof}

As already mentioned, the main rook theory contribution of this section is a closed formula for the trailing degree of a $q$-rook polynomial associated with a Ferrers diagram $\mF$. In contrast with the definition of the ``inv'' statistics (Notation~\ref{not:invrook}), which considers vertical and horizontal deletions, our characterization is expressed in terms of the diagonals of $\mF$.

\begin{theorem} \label{th:trai}
Let $\mF$ be an $n \times m$ Ferrers diagram and let $1 \le r \le \min\{n,m\}$ be an integer with $\kappa(\mF,r) \ge 1$.
We have
$$\tau(\mF,r) = \sum_{i=1}^{m+n-1} \max\{0,|D_i \cap \mF| -r\}.$$
\end{theorem}

The main ingredient in the proof of Theorem~\ref{th:trai} is the following technical result. Its role in the proof of Theorem~\ref{th:trai} will become clear later.

\begin{lemma} \label{lem:techn}
Let $\mF=[c_1, \dots, c_m]$ be an $n \times m$ Ferrers diagram with $m \ge 2$. Let $1 \le r \le \min\{n,m\}$ be an integer with $\kappa(\mF,r) \ge 1$.
Denote by $\smash{\mF'=[c_1, \dots, c_{m-1}]}$ the $\smash{c_{m-1} \times (m-1)}$ Ferrers diagram obtained from~$\mF$ by deleting the rightmost column. Moreover, 
denote the diagonals of the new matrix board of size $c_{m-1} \times (m-1)$ by $D_i' \subseteq [c_{m-1}] \times [m-1]$, for $1 \le i \le c_{m-1}+m$.
We have
\begin{multline*}
   \sum_{i=1}^{m+n-1} \min\{r,|D_i \cap \mF|\} = \\ \max \left\{ n+ \sum_{i=1}^{c_{m-1}+m-2} 
\min\{r-1,|D'_i \cap \mF'|\}, \, 
r+ \sum_{i=1}^{c_{m-1}+m-2} 
\min\{r,|D'_i \cap \mF'|\} \right\}.
\end{multline*}
\end{lemma}
\begin{proof} 
Define $I := \{1 \le i \le \min\{n,m\} \, : \, r \ge |D_i \cap \mF|\}$. The remainder of the proof contains multiple claims, which we prove separately. We start with the following one, which heavily relies on our assumption $\kappa(\mF,r) \ge 1$.

\begin{claim} \label{cl:1}
We have $|D_i \cap \mF|=i$ for all $1 \le i \le r$. 
\end{claim}
\begin{clproof}
It is enough to prove that $|D_r \cap \mF|=r$, because $\mF$ is a Ferrers diagram.
Towards a contradiction, assume that $|D_r \cap \mF|<r$ and let $(i,m-r+i) \in D_r \backslash \mF$ for some integer $1 \le i \le r$. 
Then for every $(a,b) \in [n] \times [m]$ with $a>i-1$ and $b<m-r+i+1$
we have $(a,b) \notin \mF$.
In particular, $\kappa_{i-1}(\mF,r)=0$. This is a contradiction, because $1 \le \kappa(\mF,r) \le \kappa_{i-1}(\mF,r)$. 
\end{clproof}

A straightforward consequence of Claim~\ref{cl:1} is that $|I| \ge r$. We will also need the following intermediate result.

\begin{claim} \label{cl:2}
We have
\begin{align} \label{eq:recursion}
    \sum_{i=1}^{m+n-1} \min\{r,|D_i \cap \mF|\}  \, = \sum_{i=1}^{c_{m-1}+m-1}\min\{r,|D_i \cap \mF|\}.
\end{align}
\end{claim}
\begin{clproof}
We will show that $|D_i \cap \mF|=0$ for $i \ge c_{m-1}+m$, which establishes the claim. Since 
$\mF$ is a Ferrers diagram,
for all $(a,b) \in \mF$ with $b \le m-1$ we have that $a \le c_{m-1}$. Now if $|D_i \cap \mF|>0$ for some $i \ge c_{m-1}+m$, then there exists $(a,b) \in \mF$ with $b \le m-c_{m-1}$ and $a \ge c_{m-1}+1$, yielding a contradiction.
\end{clproof}

We continue by investigating each of the two expressions in the maximum in the statement of the lemma separately. First of all, we have
\allowdisplaybreaks
\begin{align*} 
\allowdisplaybreaks
n+\sum_{i=1}^{c_{m-1}+m-2} 
&\min\{r-1,|D'_i \cap \mF'|\} \\ &= 
n+\sum_{i=2}^{c_{m-1}+m-1} 
\min\{r-1,|D'_{i-1} \cap \mF'|\} \\
&= n+\sum_{i=2}^{n} 
\min\{r-1,|D'_{i-1} \cap \mF'|\} + \sum_{i=n+1}^{c_{m-1}+m-1} \min\{r-1,|D'_{i-1} \cap \mF'|\} \\
&= n+\sum_{i=2}^{n} 
\min\{r-1,|D_i \cap \mF|-1\} + \sum_{i=n+1}^{c_{m-1}+m-1} 
\min\{r-1,|D_i \cap \mF|\} \\
&= n-(n-1)+ \sum_{i=2}^{n} 
\min\{r,|D_i \cap \mF|\} + \sum_{i=n+1}^{c_{m-1}+m-1} 
\min\{r-1,|D_i \cap \mF|\} \\
&= 1+\sum_{i=2}^{n} 
\min\{r,|D_i \cap \mF|\} + \sum_{i=n+1}^{c_{m-1}+m-1} 
\min\{r-1,|D_i \cap \mF|\}.
\end{align*}
Since $\min\{r,|D_1 \cap \mF|\}=|D_1 \cap \mF|=1$, we then obtain
\begin{align}\label{eq:cl222}
\allowdisplaybreaks
    n+\sum_{i=1}^{c_{m-1}+m-2} 
\min\{r-1,|D'_i \cap \mF'|\} \nonumber &= \sum_{i=1}^{n} 
\min\{r,|D_i \cap \mF|\} + \sum_{i=n+1}^{c_{m-1}+m-1} 
\min\{r-1,|D_i \cap \mF|\} \nonumber \\
&\le \sum_{i=1}^{c_{m-1}+m-1} 
\min\{r,|D_i \cap \mF|\} \nonumber\\  &= \sum_{i=1}^{m+n-1} 
\min\{r,|D_i \cap \mF|\},
\end{align}
where the latter equality follows from
Claim~\ref{cl:2}.

\begin{claim}\label{cl:4}
Assume that $r < \min\{n,m\}$. If $r \le i < \min\{n,m\}$ and $|D_i \cap \mF| \le r$, then $|D_{i+1} \cap \mF| \le r$. Moreover, if $|D_{\min\{n,m\}} \cap \mF| \le r$, then we have $|D_{i} \cap \mF| \le r-1$ for all $\min\{n,m\}+1 \le i \le n+m-1$.
\end{claim}
\begin{clproof}
Let $r \le i \le \min\{n,m\}-1$ and $|D_i \cap \mF| \le r$. Since $r < \min\{n,m\}$ and $|D_i \cap \mF| \le r$, we have $|D_i \backslash \mF| = t \ge 1$. Let $D_i \backslash \mF = \{(a_1,b_1),\dots, (a_t,b_t)\}$ with $a_1 < \dots < a_t$. We have 
\begin{align*}
    \{(a_1+1,b_1),\dots, (a_t+1,b_t)\} \cup \{(a_1,b_1-1),\dots, (a_t,b_t-1)\} \subseteq D_{i+1} \backslash \mF.
\end{align*}
Since $|\{a_1, a_1+1, \dots, a_t, a_t+1\}| \ge t+1$, it follows that 
\begin{align*}
    |\{(a_1+1,b_1),\dots, (a_t+1,b_t)\} \cup \{(a_1,b_1-1),\dots, (a_t,b_t-1)\}| \ge t+1.
\end{align*}
Therefore $| D_{i+1} \backslash \mF| \ge t+1$, hence
\begin{align} \label{eq:cl4}
    i+1-|D_{i+1}\cap \mF| = |D_{i+1}\backslash \mF| \ge |D_{i}\backslash \mF| +1 = i-|D_{i}\cap \mF|+1 \ge i-r+1,
\end{align}
where we used that $|D_i|=i$ and that $|D_{i+1}|=i+1$. Rewriting~\eqref{eq:cl4} proves the first statement in the claim.

For the second part, suppose that $|D_{\min\{n,m\}} \cap \mF| \le r$ and write $|D_{\min\{n,m\}} \backslash \mF| = \{(a_1,b_1),\dots, (a_t,b_t)\} \ne \emptyset$ with $a_1 < \dots < a_t$. Then following the same reasoning as before one shows that $| D_{\min\{n,m\}+1} \backslash \mF| \ge t+1$ and, similar to~\eqref{eq:cl4}, that 
\begin{align*}
    \min\{n,m\}-|D_{\min\{n,m\}+1}\cap \mF| \ge \min\{n,m\}-r+1,
\end{align*}
where we used that $|D_{\min\{n,m\}+1}| \le \min\{n,m\}$. This shows that $|D_{\min\{n,m\}+1} \cap \mF| \le r-1$. In an analogous way one proves that $|D_{i} \cap \mF| \le r-1$ for all $\min\{n,m\}+1 \le i \le n+m-1$.
\end{clproof}

The following claim gives a sufficient condition for when the bound in~\eqref{eq:cl222} is attained with equality.
\begin{claim}\label{cl:3}
If $|I| > r$, then $|D_i \cap \mF| \le r-1$ for all $\min\{n,m\}+1 \le i \le m+n-1$.
In particular,
$|D_i \cap \mF| \le r-1$ for all $n+1 \le i \le c_{m-1}+m-1$.
\end{claim}
\begin{clproof}
Let $|I| = s > r$. Clearly we then have
$r< \min\{n,m\}$.
Write $I=\{i_1, \dots, i_s\}$ with $i_1 <i_2 < \dots < i_s \le \min\{n,m\}$.
By Claim~\ref{cl:1} we have $i_j=j$ for all $j \in \{1, \dots,r\}$. In particular, $r<i_s$. Note that by definition of $I$ we have $|D_{i_s} \cap \mF| \le r$ and by Claim~\ref{cl:4} this implies that $i_s=\min\{n,m\}$. Using the second part of Claim~\ref{cl:4} we also conclude that $|D_i \cap \mF| \le r-1$ for all $\min\{n,m\}+1 \le i \le m+n-1$. 
\end{clproof}

If $|I| > r$, then by Claim~\ref{cl:3} we 
have equality in~\eqref{eq:cl222}, which means
\begin{align} \label{eq:maxx}
    n+\sum_{i=1}^{c_{m-1}+m-2} \min\{r-1,|D'_i \cap \mF'|\}= \sum_{i=1}^{m+n-1} 
\min\{r,|D_i \cap \mF|\}.
\end{align}
Note moreover that if $i \in I\backslash\{1\}$, then we have $\min\{r,|D_i \cap \mF|\} = \min\{r,|D_{i-1}' \cap \mF'|\}+1$ and if $i \in [n] \backslash I$ then $\min\{r,|D_i \cap \mF|\} = \min\{r,|D_{i-1}' \cap \mF'|\}$. Furthermore, for all $i \ge n+1$
we have $\min\{r,|D_i \cap \mF|\} = \min\{r,|D_{i-1}' \cap \mF'|\}$. Therefore
\allowdisplaybreaks
\begin{align*} 
    \sum_{i=1}^{m+n-1} &\min\{r,|D_i \cap \mF|\}  \\ &= \sum_{i=1}^{c_{m-1}+m-1}\min\{r,|D_i \cap \mF|\}  \\
    &= \sum_{i\in I}\min\{r,|D_i \cap \mF|\} + \sum_{i\in [n] \backslash I}\min\{r,|D_i \cap \mF|\} + \sum_{i=n+1}^{c_{m-1}+m-1}\min\{r,|D_i \cap \mF|\} \\
    &= 1+\sum_{i\in I\backslash\{1\}}\left(\min\{r,|D_{i-1}' \cap \mF'|\}+1\right) + \sum_{i\in [n] \backslash I}\min\{r,|D_{i-1}' \cap \mF'|\} \\ & \hspace{7.8cm}  +\sum_{i=n+1}^{c_{m-1}+m-1}\min\{r,|D_{i-1}' \cap \mF'|\} \\
    &= |I| + \sum_{i=1}^{c_{m-1}+m-2}\min\{r,|D_{i}' \cap \mF'|\},
\end{align*}
where the first equality follows from Claim~\ref{cl:2}.
In particular,
\begin{align*}
    \sum_{i=1}^{m+n-1} \min\{r,|D_i \cap \mF|\} \ge  r+ \sum_{i=1}^{c_{m-1}+m-2}\min\{r,|D'_i \cap \mF'|\},
\end{align*}
with equality if and only if $|I|=r$. Together with~\eqref{eq:maxx}, this concludes the proof.
\end{proof}

We are now ready to establish 
Theorem~\ref{th:trai}.

\begin{proof}[Proof of Theorem~\ref{th:trai}]
Since the union of the diagonals $D_1, \ldots, D_{m+n-1}$
is the entire matrix board $[n] \times [m]$, we have 
\begin{equation} \label{hag2}
|\mF|-\sum_{i=1}^{m+n-1} \min\{r,|D_i \cap \mF|\} = \sum_{i=1}^{m+n-1} \max\{0,|D_i \cap \mF|-r\}.
\end{equation}
Therefore, by Theorem~\ref{degHag}, proving Theorem~\ref{th:trai} is equivalent to proving that
\begin{equation}
    \deg(P_q(\mF,r)) = \sum_{i=1}^{m+n-1} \min\{r,|D_i \cap \mF|\}.
\end{equation}
It follows from~\cite[Theorem 7.1]{gluesing2020partitions} that
for an $n \times m$ Ferrers diagram $\mF=[c_1,\ldots,c_m]$
the quantity $\deg(P_q(\mF,r))$
is uniquely determined by the recursion
\begin{equation} \label{recc} \deg(P_q(\mF,r)) = 
\max\biggl\{  
n+\deg(P_q(\mF',r-1)), \; r+\deg(P_q(\mF',r))
\biggr\},
\end{equation}
where $\mF'=[c_1,\ldots,c_{m-1}]$,
with initial conditions:
\begin{equation} \label{casess}
    \begin{cases}
\deg(P_q(\mF,0))=0 & \mbox{ for all Ferrers diagrams $\mF$}, \\
\deg(P_q(\mF,1))=c_1 & \mbox{ if $\mF=[c_1]$ is a $c_1 \times 1$ Ferrers diagram,} \\
\deg(P_q(\mF,r))=-\infty & \mbox{ if $\mF$ is a $c_1 \times 1$ Ferrers diagram and $r \ge 2$.} 
\end{cases}
\end{equation}
By Lemma~\ref{lem:techn}, the quantity
$$\Delta_q(\mF,r):= \begin{cases}
-\infty & \mbox{if $\kappa(\mF,r)=0$,} \\
\sum_{i=1}^{m+n-1} \min\{r,|D_i \cap \mF|\} & \mbox{otherwise,}
\end{cases}$$ 
satisfies
the recursion and the initial conditions
in~\eqref{recc} and~\eqref{casess}, respectively.
Therefore it must be that
$\deg(P_q(\mF,r)) = 
\Delta_q(\mF,r)$ for every $\mF$ and $r$, which proves the theorem.
\end{proof}

We can now return to the combinatorial characterization of MDS-constructible pairs.
We start by observing the following.

\begin{proposition} \label{thm:mdsbound}
Let $\mF$ be an $n \times m$ Ferrers diagram and let $1 \le d \le \min\{n,m\}$ be an integer. We have
$$\kappa(\mF,d) \ge \sum_{i=1}^{m+n-1} \max\{0, |D_i \cap \mF|-d+1\}.$$
\end{proposition}
\begin{proof}
Let $0 \le j \le d-1$ be an integer such that $\kappa_j(\mF,d)=\kappa(\mF,d)$.
Denote by $\mF_j$ the subset of $\mF$ made by those points that are not contained in the topmost $j$ rows of $\mF$, nor in its rightmost $d-1-j$ columns. 
We have $|D_i \cap \mF_j| \ge \max\{|D_i \cap \mF| -d+1,0\}$ for all $1 \le i \le m+n-1$. Summing these inequalities over $i$ gives
\begin{equation} \label{bbb}
\kappa(\mF,d) = \sum_{i=1}^{m+n-1} |D_i \cap \mF_j| \ge \sum_{i=1}^{m+n-1} \max\{0,|D_i \cap \mF| -d+1\},
\end{equation}
where the first equality in~\eqref{bbb} follows from the fact that the diagonals are disjoint and their union is $\mF$. 
\end{proof}

The bound of Proposition~\ref{thm:mdsbound} is not sharp in general.

\begin{example}
The value of $\tau(\mF,d-1)$ for the $5\times 6$
Ferrers diagram $\mF=[5,5,5,5,5,5]$ and $d=4$ is 
$\tau(\mF,3)=6$. 
Note that we have $\kappa(\mF,4) = 12>6=\tau(\mF,3)$.
\end{example}

The following theorem shows that when defining MDS-constructible pairs one can consider the sum over all the diagonals, if $d \ge 2$.
This gives us a characterization of MDS-constructible pairs that is symmetric in $n$ and $m$; see
Remark~\ref{rmk:symm}.

\begin{theorem} \label{prop:newmdsconstr}
Let $\mF$ be an $n \times m$ Ferrers diagram with $m \ge n$ and let $2 \le d \le n$ be an integer. Then the following are equivalent:
\begin{enumerate}
    \item $\smash{\kappa(\mF,d) = \sum_{i=1}^m \max\{0, |D_i \cap \mF|-d+1\}}$, i.e., the pair $(\mF,d)$ is MDS-constructible in the sense of Definition~\ref{def:mdsconstr};
    \item $\smash{\kappa(\mF,d) = \sum_{i=1}^{m+n-1} \max\{0, |D_i \cap \mF|-d+1\}}$.
\end{enumerate}
\end{theorem}

\begin{proof} 
By Proposition~\ref{thm:mdsbound}, we have
$$\sum_{i=1}^m \max\{|D_i \cap \mF|-d+1,0\} \le 
\sum_{i=1}^{m+n-1} \max\{0,|D_i \cap \mF|-d+1\} \le \kappa(\mF,d).$$
Thus $(\mF,d)$ being MDS-constructible implies  $\kappa(\mF,d) = \sum_{i=1}^{m+n-1} \max\{0, |D_i \cap \mF|-d+1\}$.
For the other direction, we need to show that if $\kappa(\mF,d) = \sum_{i=1}^{m+n-1}\max\{0,|D_i \cap \mF|-d+1\}$, then $\kappa(\mF,d) = \sum_{i=1}^{m}\max\{0,|D_i \cap \mF|-d+1\}$ as well. We proceed by contradiction and suppose 
that $$\kappa(\mF,d) = \sum_{i=1}^{m+n-1}\max\{|D_i \cap \mF|-d+1,0\} > \sum_{i=1}^{m}\max\{|D_i \cap \mF|-d+1,0\}.$$ Then there exists a diagonal $D_u$, for some $u \ge m+1$, with $|D_u \cap \mF|-d+1 > 0$. 
Let $0 \le j \le d-1$ be such that
$\kappa(\mF,d) = \kappa_j(\mF,d)$.
Denote by $\mF_j$ the subset of $\mF$ made by those points that are not contained in the topmost $j$ rows of $\mF$, nor in its rightmost $d-1-j$ columns. Then 
$|\mF_j|=\kappa(\mF,d)$
and as in the proof of Proposition~\ref{thm:mdsbound} we have $|D_i \cap \mF_j| \ge
\max\{0,|D_i \cap \mF|-d+1\}$ for all $1 \le i \le m+n-1$.  Summing over $i$ gives $$\kappa(\mF,d)=\sum_{i=1}^{m+n-1} |D_i \cap \mF_j| \ge 
\sum_{i=1}^{m+n-1} \max\{0,|D_i \cap \mF|-d+1\} = \kappa(\mF,d),$$
where the first equality follows from the definition of $\mF_j$ and the latter equality is by assumption.
Since $|D_i \cap \mF_j| \ge \max\{0,|D_i \cap \mF|-d+1\}$ for all $1 \le i \le m+n-1$, this implies $|D_i \cap \mF_j| = \max\{0,|D_i \cap \mF|-d+1\}$ for all $1 \le i \le m+n-1$. In particular,  since
$|D_u \cap \mF|-d+1 > 0$ by assumption, we must have
$|D_u \cap \mF_j| = |D_u \cap \mF|-d+1$. 
This implies that $D_u \cap \mF$ contains $d-1$ entries that belong to the topmost $j$ rows and rightmost $d-1-j$ columns of $\mF$. This is however a contradiction, because $m \ge n$, $d \ge 2$, and $u \ge m+1$, again by assumption.
\end{proof}

In view of Theorem~\ref{prop:newmdsconstr}, we propose the following slightly modified definition of MDS-constructible pair, which coincides with the one of~\cite{antrobus2019maximal} when $m \ge n$ and $d \ge 2$. This addresses the points discussed in Remark~\ref{rmk:symm}. Notice that according to our definition the pair $(\mF,1)$ is always MDS-constructible, while it might not be according to Definition~\ref{def:mdsconstr}.

\begin{definition}[updates Definition~\ref{def:mdsconstr}] \label{def:updates}
Let $\mF$ be an $n \times m$ Ferrers diagram and let $1 \le d \le \min\{n,m\}$ be an integer. The pair $(\mF,d)$ is called \textbf{MDS-constructible} if 
$$\kappa(\mF,d) = \sum_{i=1}^{m+n-1} \max\{0,|D_i \cap \mF|-d+1\}.$$
\end{definition}

By combining Theorem~\ref{th:trai} with Theorem~\ref{prop:newmdsconstr}, we finally obtain a pure rook theory characterization of MDS-constructible pairs.

\begin{corollary} \label{cor:main}
Let $\mF$ be an $n \times m$ Ferrers diagram and let $1 \le d \le \min\{n,m\}$ be an integer with $\kappa(\mF,d) \ge 1$.
The following are equivalent:
\begin{enumerate}
    \item The pair $(\mF,d)$ is MDS-constructible, according to Definition~\ref{def:updates}; 
    \item $\kappa(\mF,d) = \tau(\mF,d-1)$.
\end{enumerate}
\end{corollary}

In words, Corollary~\ref{cor:main} states that the construction of 
\cite{gorla2017subspace,etzion2016optimal,roth1991maximum}
is optimal if and only if
the Etzion-Silberstein Bound of Theorem~\ref{thm:dimbound} coincides
with the trailing degree of the $(d-1)$th $q$-rook polynomial of $\mF$.

\section{Asymptotics of the Etzion-Silberstein Conjecture}
\label{sec:3}

In this section we solve a problem that can be regarded as the 
``asymptotic'' analogue of the Etzion-Silberstein Conjecture for $q \to +\infty$ (see Problem~\ref{probb} below for a precise statement). As we will see, this problem has again a strong connection with rook theory.
In the remainder of the paper we will use the notion of MDS-constructible pair introduced in Definition~\ref{def:updates}.

\begin{notation}
We denote by
$$\qbin{a}{b}{q}= \prod_{i=0}^{b-1}\frac{\left(q^a-q^i\right)}{\left(q^b-q^i\right)}$$ be the $q$-binomial coefficient of integers $a \ge b \ge 0$, which counts the number of $b$-dimensional subspaces of an $a$-dimensional space over $\F_q$; see e.g.~\cite{stanley2011enumerative}. We will also  use the standard Bachmann-Landau notation (“Big O”, “Little O”, and~“$\sim$”) to describe the asymptotic growth of real-valued functions; see for example~\cite{de1981asymptotic}. If $Q$ denotes the set of prime powers, we omit ``$q \in Q$'' when writing $q \to +\infty$.
\end{notation}

In the remainder of this paper we will repeatedly need the following asymptotic estimate 
for the $q$-binomial coefficient: 
\begin{equation} \label{eq:qbin}
\qbin{a}{b}{q} \sim q^{b(a-b)} \quad \mbox{as $q \to +\infty$},
\end{equation}
for all
integers $a \ge b \ge 0$.
We will apply this well-known fact throughout the paper without explicitly
referring to it.

When studying the Etzion-Silberstein Conjecture in the asymptotic regime, we are interested in 
the asymptotic behavior, as $q \to +\infty$,
of the proportion of optimal~$[\mF,d]_q$-spaces among all spaces having the same dimension. This motivates the following definition.

\begin{definition} \label{def:delta}
Let $\mF$ be an $n \times m$ Ferrers diagram. 
For $1 \le k \le |\mF|$
and $1 \le d \le \min\{n,m\}$, let $$\delta_q(\mF, k, d):= \frac{|\{\mC \le \mat \, : \,  \mC \mbox{ is an $[\mF,d]_q$-space}, \; \dim(\mC)=k\}|}{\qbin{|\mF|}{k}{q}}$$
denote the \textbf{density} (\textbf{function}) of  $[\mF,d]_q$-spaces among all $k$-dimensional subspaces of $\F_q[\mF]$. Their \textbf{asymptotic density} as $q \to +\infty$ is
$\lim_{q \to + \infty} \delta_q(\mF, k, d)$, when the limit exists. Moreover, when the asymptotic density tends to 1 (as $q \to +\infty$), we say the corresponding spaces are \textbf{dense}; if it tends to 0, we say that they are \textbf{sparse}.
\end{definition}

The following problem can be viewed as the ``asymptotic'' analogue of the Etzion-Silberstein Conjecture.

\begin{problem} \label{probb}
Fix an $n \times m$ Ferrers diagram $\mF$ and an integer $1 \le d \le \min\{n,m\}$.
Determine for which values of $1 \le k \le |\mF|$ we have
$\lim_{q \to +\infty} \delta_q(\mF,k,d)=0$ and for which values we have $\lim_{q \to +\infty} \delta_q(\mF,k,d)=1$. Determine 
the value of $\lim_{q \to +\infty} \delta_q(\mF,\kappa(\mF,d),d)$.
\end{problem}

Problem~\ref{probb} has been proposed and solved in~\cite{antrobus2019maximal} for some classes of pairs $(\mF,d)$.
The two main results of~\cite{antrobus2019maximal} in this context are the following.

\begin{theorem}[see \textnormal{\cite[Theorem VI.8]{antrobus2019maximal}}] \label{thm:heide1}
Let $\mF$ be an $n \times m$ Ferrers diagram with $m \ge n$ and let $1 \le d \le n$ be an integer. If $(\mF,d)$ is MDS-constructible and 
$\kappa=\kappa(\mF,d)$, then  $\lim_{q \to +\infty} \delta_q(\mF,\kappa,d) = 1$.
\end{theorem}

\begin{theorem}[see \textnormal{\cite[Corollary VI.13]{antrobus2019maximal}}] \label{thm:heide2}
Let $\mF$ be an $n \times m$ Ferrers diagram with $m \ge n$ and let $\kappa=\kappa(\mF,n)$. Then  $(\mF,n)$ is MDS-constructible if and only if $\lim_{q \to +\infty} \delta_q(\mF,\kappa,n) = 1$.
\end{theorem}

Both Theorem~\ref{thm:heide1} and Theorem~\ref{thm:heide2} were established in~\cite{antrobus2019maximal}
by using arguments based on the 
algebraic closure of $\F_q$.
As the reader will soon notice, the approach taken in this paper is of a completely different nature.
It will allow us to generalize~\cite[Corollary VI.13]{antrobus2019maximal} and to 
solve Problem~\ref{probb} completely,
answering the question stated in~\cite[Open Problem~(a)]{antrobus2019maximal}; see Corollary~\ref{cor:ans} below.

The next result shows that Problem~\ref{probb}
exhibits a very strong connection with rook theory.
More in detail, it proves that
the decisive dimension for sparseness and density 
is precisely the trailing degree of the $(d-1)$th $q$-rook polynomial associated with $\mF$.

\begin{theorem} \label{thm:sparsedense}
Let $\mF$ be an $n \times m$ Ferrers diagram and let $2\le d \le n$ and $1 \le k \le |\mF|$ be integers
with $\kappa(\mF,d) \ge 1$.  The following hold.  
\begin{enumerate}
    \item If $k \le \tau(\mF,d-1)$, then $\lim_{q\to +\infty}\, \delta_q(\mF,k,d) = 1.$
    \item If $k \ge  \tau(\mF,d-1)+2$, then $\lim_{q\to +\infty}\, \delta_q(\mF,k,d) = 0.$
    \item If $k =  \tau(\mF,d-1)+1$, then $\limsup_{q\to +\infty}\, \delta_q(\mF,k,d) \le 1/2.$
\end{enumerate}
\end{theorem}

The proof of Theorem~\ref{thm:sparsedense} relies on~\cite[Theorem 1]{haglund} and on the machinery developed in~\cite{gruica2020common} to estimate the asymptotic density of (non-)isolated vertices in bipartite graphs. We start by introducing the needed terminology.

\begin{definition}\label{def:ball}
The \textbf{ball} of \textbf{radius}
$r$ in $\F_q[\mF]$ is the set
\begin{align*}
    B_q(\mF, r):= \bigcup_{i=0}^r \{M \in \F_q[\mF] : \rk(M)=i\}.
\end{align*}
The cardinality of 
$B_q(\mF, r)$
is denoted by 
$\bbq{\mF, r}$ in the sequel. Note that, by definition, we have
$\bbq{\mF, r} = \sum_{i=0}^r P_q(\mF,i)$, where $P_q(\mF,i)$ is defined in Notation~\ref{notPq}.
\end{definition}

A formula for $\bbq{\mF, r}$ for all $\mF$ and $r$ can be found in~\cite[Theorem 5.3]{gluesing2020partitions}. We will also need the following result.

\begin{corollary} \label{degrees}
Let $\mF$ be an $n \times m$ Ferrers diagram and let $r \ge 0$ be an integer.
Then $P_q(\mF,r)$ and $\bbq{\mF,r}$ are polynomials in $q$ of the same degree. In particular,
$$\tau(\mF,r) = |\mF| - \deg(\bbq{\mF,r}).$$
\end{corollary}

\begin{proof}
It is well known that $P_q(\mF,r)$ is a polynomial in $q$; see e.g.~\cite{haglund,gluesing2020partitions}. Therefore, by definition,
$\bbq{\mF,r}$ is a polynomial in $q$ as well. The fact that $P_q(\mF,r)$ and $\bbq{\mF,r}$ have the same degree follows from~\cite[Theorem 7.13]{gluesing2020partitions}. The last identity in the statement directly follows from Theorem~\ref{degHag}.
\end{proof}

We are now ready to establish Theorem~\ref{thm:sparsedense}.

\begin{proof}[Proof of Theorem~\ref{thm:sparsedense}]
By~\cite[Theorem 4.2 (1)]{gruica2020common}, if $\deg(\bbq{\mF,d-1}) \le |\mF|-k$ then we have $\lim_{q\to +\infty}\, \delta_q(\mF,k,d) = 1.$  From Corollary~\ref{degrees} it follows that we have $\deg(\bbq{\mF,d-1}) \le |\mF|-k$ if and only if $k \le \tau(\mF,d-1).$ Analogously, one can use~\cite[Theorem 4.2 (2)]{gruica2020common} to show that if $k \ge  \tau(\mF,d-1)+2$, then $\lim_{q\to +\infty}\, \delta_q(\mF,k,d) = 0.$ Finally, by~\cite[Proposition 4.4]{gruica2020common} we have $\limsup_{q\to +\infty}\, \delta_q(\mF,k,d) \le 1/2$ when $k = \tau(\mF,d-1)+1$.
\end{proof}

Theorem~\ref{thm:sparsedense} generalizes the sparseness of MRD codes established in~\cite{gruica2020common}, as the following remark illustrates.

\begin{remark}
In the special case where $\mF=[m,\dots,m]$ is an $n \times m$ Ferrers diagram, the optimal $[\mF,d]_q$-spaces correspond exactly to $n \times m$ MRD codes of minimum distance $d$. Note that for $n \ge 2$ we always have $m(n-d+1)=\kappa(\mF,d) \ge  \sum_{i=1}^{m+n-1}\max\{0,|D_i \cap \mF|-d+1\}+2=(m-d+1)(n-d+1)+2$. 
Therefore, by Theorem~\ref{th:trai},  Theorem~\ref{thm:sparsedense} generalizes the sparseness of MRD codes established in~\cite{gruica2020common}.
\end{remark}

Note that in the third scenario of Theorem~\ref{thm:sparsedense} 
we do not have sparseness in general (even though we have non-density). The following example based on~\cite{antrobus2019maximal} shows that
$\delta_q(\mF,k,d)$ can indeed converge to a positive constant when $k=\tau(\mF,d-1)$.

\begin{example}
Consider the $2 \times m$ Ferrers diagram $\mF=[2,\dots, 2]$ with $m \ge 2$ and let $d=2$.
We have $\kappa(\mF,2) = m$ and, by Theorem~\ref{th:trai},  $\tau(\mF,1)=\sum_{i=1}^{m+1}\max\{0,|D_i \cap \mF|-1\}=m-1$. It follows from \cite[Corollary VII.6]{antrobus2019maximal} that $$0<\lim_{q\to +\infty}\, \delta_q(\mF,m,2) = \left(\sum_{j=0}^m \frac{(-1)^j}{j!}\right)^{(d-1)(n-d+1)} <1.$$ In particular, optimal $[\mF,2]_q$-spaces are neither sparse, nor dense.
\end{example}

Combining 
Theorem~\ref{prop:newmdsconstr}
with 
Theorem~\ref{thm:sparsedense} we obtain the following corollary, which shows that
optimal $[\mF,d]_q$-spaces are dense as $q \to +\infty$ if and only if the pair $(\mF,d)$ is MDS-constructible. This also  answers~\cite[Open Problem~(a)]{antrobus2019maximal}.

\begin{corollary} \label{cor:ans}
Let $\mF$ be an $n \times m$ Ferrers diagram and let $1 \le d\le n$ be an integer with $\kappa(\mF,d)\ge 1$. 
The following are equivalent:
\begin{enumerate}
    \item $\lim_{q \to +\infty} \delta(\mF,\kappa(\mF,d),d) =1$;
    \item $(\mF,d)$ is MDS-constructible.
\end{enumerate}
\end{corollary}

We conclude this section with the following observation.

\begin{remark}
Corollary~\ref{cor:ans} shows 
that the Etzion-Silberstein Conjecture 
holds whenever the pair~$(\mF,d)$ is MDS-constructible and $q$ is large, with a proof that does not depend on the construction of~\cite{roth1991maximum,etzion2016optimal,gorla2017subspace}.
\end{remark}


\section{Existence Results}
\label{sec:existence}
\label{sec:4}

In this very short section we show 
how some (new) instances of the
Etzion-Silberstein Conjecture can be established in a non-constructive way
using a purely combinatorial argument.
We will obtain existence of optimal $[\mF,d]_q$-spaces for some MDS-constructible pairs $(\mF,d)$ and for values of $q$ that are smaller than the minimum value needed in Theorem~\ref{construc}.

\begin{proposition} \label{prop:existlb}
Let $\mF$ be an $n \times m$ Ferrers diagram and let $2 \le d \le \min\{n,m\}$ and $1 \le k \le \kappa(\mF,d)$ be integers. 
There are at least
\begin{align} \label{eq:lb}
    \qbin{|\mF|}{k}{q}-\frac{\bbq{\mF, d-1}-1}{q-1}\qbin{|\mF|-1}{k-1}{q}
\end{align}
$k$-dimensional $[\mF,d]_q$-spaces. In particular, if the expression in~\eqref{eq:lb} is larger than 0, then such a space exists.
\end{proposition}
\begin{proof}
The $[\mF,d]_q$-spaces of dimension $k$ are
the $k$-dimensional subspaces of $\F_q[\mF]$ that intersect $\bbq{\mF, d-1}$ trivially, i.e., in $\{0\}$. We denote by $\mathfrak{B}(\mF, d-1)$ a set of representatives (up to $\F_q$-multiples) of the nonzero elements of $\bbq{\mF, d-1}$. We have $|\mathfrak{B}(\mF, d-1)| = (\bbq{\mF, d-1}-1)/(q-1)$. Counting the cardinality of the set 
\begin{align*}
    \{(\mC,B) \, : \,  B \in \mathfrak{B}(\mF, d-1), \; \mC \le \F_q[\mF], \; \dim(\mC)=k, \; B \in \mC\}
\end{align*}
in two ways we obtain
\begin{align*}
    |\mathfrak{B}(\mF, d-1)| \qbin{|\mF|-1}{k-1}{q} &= \sum_{B \in \mathfrak{B}(\mF, d-1)} |\{\mC \le \F_q[\mF] \, : \,  \dim(\mC)=k, \; B \in \mC\}| \\
    & = \sum_{\substack{\mC \le \F_q[\mF], \\ \dim(\mC)=k}} |\{B \in \mathfrak{B}(\mF, d-1) \, : \,  B \in \mC\}| \\
    &\ge |\{\mC \le \F_q[\mF] \, : \, \dim(\mC)=k, \, \mbox{$\mC$ is not an $[\mF,d]_q$-space}\}|.
\end{align*}
The lower bound in the statement is a direct consequence of the above inequality and of the definition of $q$-binomial coefficient.
\end{proof}

\begin{example} \label{ex:easyexist}
Consider the $5 \times 6$ Ferrers diagram in Figure~\ref{F-233345}.
We have $|\mF|=20$ and for $d=4$ the pair $(\mF,d)$ is MDS-constructible. By the construction described in Remark~\ref{rem:constr}, we know that for $q \ge \max\{|D_i \cap \mF| : 1 \le i \le 6\}-1 = 4$ there exists an optimal $[\mF,d]_q$-space 
of dimension $\kappa(\mF,4)=3$. However,
for $q=3$ we have
$\bbq{\mF,3} = 243679185$
(where we computed the value using \cite[Theorem 5.3]{gluesing2020partitions}) and thus 
by Proposition~\ref{prop:existlb} we have 
\begin{align*}
\qbin{20}{3}{3}-\frac{(243679185-1)}{2}\qbin{19}{2}{3} = 345241120940998775695104,
\end{align*}
meaning that there are at least 345241120940998775695104 optimal 
$[\mF,d]_q$-spaces already for $q=3$.
The lower bound for $q=2$ is however negative (its value is $-6510288900541266$).
\end{example}

     \begin{figure}[ht] 
    \centering
     {\small
     \begin{tikzpicture}[scale=0.4]
         \draw (5.5,1.5) node (b1) [label=center:$\bullet$] {};
         \draw (5.5,5.5) node (b1) [label=center:$\bullet$] {};
         \draw (5.5,4.5) node (b1) [label=center:$\bullet$] {}; 
         \draw (5.5,3.5) node (b1) [label=center:$\bullet$] {};
         \draw (5.5,2.5) node (b1) [label=center:$\bullet$] {};

         \draw (4.5,2.5) node (b1) [label=center:$\bullet$] {};
         \draw (4.5,5.5) node (b1) [label=center:$\bullet$] {};
         \draw (4.5,4.5) node (b1) [label=center:$\bullet$] {}; 
         \draw (4.5,3.5) node (b1) [label=center:$\bullet$] {};
         
         \draw (3.5,3.5) node (b1) [label=center:$\bullet$] {};
         \draw (3.5,5.5) node (b1) [label=center:$\bullet$] {};
         \draw (3.5,4.5) node (b1) [label=center:$\bullet$] {}; 

         \draw (2.5,3.5) node (b1) [label=center:$\bullet$] {};
         \draw (2.5,5.5) node (b1) [label=center:$\bullet$] {};
         \draw (2.5,4.5) node (b1) [label=center:$\bullet$] {};   
         
         \draw (1.5,3.5) node (b1) [label=center:$\bullet$] {};
         \draw (1.5,5.5) node (b1) [label=center:$\bullet$] {};
         \draw (1.5,4.5) node (b1) [label=center:$\bullet$] {};
         
         \draw (0.5,5.5) node (b1) [label=center:$\bullet$] {};
         \draw (0.5,4.5) node (b1) [label=center:$\bullet$] {};

     \end{tikzpicture}
     }
     \caption{The Ferrers diagram $\mF=[2,3,3,3,4,5]$ for Example~\ref{ex:easyexist}.}
     \label{F-233345}
     \end{figure}
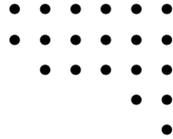

Table~\ref{table_exist} collects some examples of parameter sets
for which 
Proposition~\ref{prop:existlb} gives existence of optimal $[\mF,d]$-spaces for small values of $q$.

\begin{table}[h!]
    \centering
    \renewcommand\arraystretch{1.2}
 \begin{tabular}{|c|c|c|c|c|c|} 
 \hline
 $\mF$ & $d$ & $\kappa(\mF,d)$ & $q$ & lower bound  & constructible as in Remark~\ref{rem:constr}\\\noalign{\global\arrayrulewidth 1.8pt}
    \hline
    \noalign{\global\arrayrulewidth0.4pt}
  $[ 2, 3, 4, 4, 4, 5, 6 ]$ & 5& 3 & 3 & $1.06 \cdot 10^{33}$ & for $q \ge 4$\\ 
  \hline 
  $[ 2, 2, 2, 2, 3, 4, 5, 6, 7 ]$ & 3& 15 & 3  & $1.79 \cdot 10^{128}$ & for $q \ge 7$\\ 
  \hline 
  $[ 2, 3, 4, 4, 5, 5, 5, 5, 5, 7 ]$ & 6 & 2 & 2  & $2.92 \cdot 10^{24}$ & for $q \ge 5$\\ 
  \hline 
  $[ 1, 3, 4, 6, 6, 6, 6, 7, 8 ]$ & 7 & 3 & 4  & $1.1 \cdot 10^{79}$ & for $q \ge 7$\\ 
  \hline 
  $[ 4, 6, 6, 6, 7, 7, 7, 8, 9 ]$ & 8 & 3 & 5  & $6.19 \cdot 10^{118}$ & for $q \ge 8$\\ 
  \hline
\end{tabular}
\caption{Examples illustrating Proposition~\ref{prop:existlb}.
Note that the bounds on $q$ in the last column of the table are based on the assumption that the MDS conjecture holds, i.e. that there exists an $[n,k]_q$ MDS code if and only if $n \le q+1$ for all prime powers $q$ and integers $2 \le k \le q-1$, except when $q$ is even and $k\in \{3,q-1\}$, in which case $n\le q+2$; see~\cite{segre1955curve}.}
    \label{table_exist}
\end{table}

Even though optimal $[\mF,d]_q$-spaces are dense when $(\mF,d)$ is an MDS-constructible pair (Theorem~\ref{thm:sparsedense}),
we conclude this section by showing that the optimal spaces actually obtained from MDS codes via the construction illustrated in  Remark~\ref{rem:constr} are very few.

\begin{notation} \label{not:mdsconstr}
Let $\mF$ be an $n \times m$ Ferrers diagram with $m \ge n$ and let $2 \le d \le n$ be an integer for which $(\mF,d)$ is MDS-constructible. We denote by 
$\mbox{MC}(\mF,d)$
the set of codes constructed as in Remark~\ref{rem:constr}.
\end{notation}

\begin{proposition} \label{thm:lbmdscon}
Let $\mF$ be an $n \times m$ Ferrers diagram with $m \ge n$ and let $2 \le d \le n$ be an integer such that $(\mF,d)$ is MDS-constructible. We have  
\begin{align*}
    |\mbox{MC}(\mF,d)| \le \prod_{j=1}^{\ell} \qbin{n_{i_j}}{n_{i_j}-d+1}{q},
\end{align*}
where the set $I=\{i_1,\dots,i_{\ell}\}$ is defined as in Remark~\ref{rem:constr} and $n_{i_j}=|D_{i_j} \cap \mF|$  for all $i_j \in I$.
\end{proposition}
\begin{proof}
The statement of the theorem easily follows by overestimating the number of linear MDS codes in $\smash{\F_q^{n_{i_j}}}$ of minimum distance $d$ by the number of linear spaces of dimension $n_{i_j}-d+1$ over $\F_q$, for all $i_j \in I$.
\end{proof}

The following result gives the asymptotic density of optimal $[\mF,d]$-spaces constructed as in Remark~\ref{rem:constr} within the set of spaces with the same parameters, where we make use of Theorem~\ref{thm:sparsedense}.

\begin{corollary} \label{cor:constrsparse}
Let $\mF$ be an $n \times m$ Ferrers diagram with $m \ge n$ and let $1 \le d \le n$ be an integer such that
$(\mF,d)$ is MDS-constructible.
Let $\kappa=\kappa(\mF,d)$. We have
\begin{align*}
    \frac{|\mbox{MC}(\mF,d)|}{\qbin{|\mF|}{\kappa}{q}\delta_q(\mF,\kappa,d)} \in O\left(q^{-\kappa(|\mF|-\kappa-d+1)}\right) \quad \textnormal{as $q \to +\infty$}.
\end{align*}
\end{corollary}

\begin{proof}
We combine the upper bound for $|\mbox{MC}(\mF,d)|$ in Proposition~\ref{thm:lbmdscon} with the fact that if~$(\mF,d)$ is MDS-constructible then we have $\lim_{q \to +\infty} \delta_q(\mF,\kappa,d) = 1$;
see Theorem~\ref{thm:sparsedense}. We let $I=\{i_1,\dots,i_{\ell}\}$ and $n_{i_j}=|D_{i_j} \cap \mF|$ for all $i_j \in I$ as in Proposition~\ref{thm:lbmdscon}. It holds that 
\begin{align*}
    \qbin{n_{i_j}}{n_{i_j}-d+1}{q} \sim q^{(n_{i_j}-d+1)(d-1)} \quad \textnormal{as $q \to +\infty$}
\end{align*}
for all $i_j \in I$. Therefore we have $\smash{|\mbox{MC}(\mF,d)| \in O\left(q^{(d-1)\kappa}\right)}$ as $q \to +\infty$,
where we used the fact that  $\smash{\sum_{j=1}^{\ell}(n_{i_j}-d+1)(d-1) = (d-1)\kappa}$, since $(\mF,d)$ is MDS-constructible. Therefore, we have 
\begin{equation*}
    \frac{|\mbox{MC}(\mF,d)|}{\qbin{|\mF|}{\kappa}{q}} \in O\left(q^{-\kappa(|\mF|-\kappa-d+1)}\right) \quad \textnormal{ as $q \to +\infty$}, 
\end{equation*}
which concludes the proof.
\end{proof}

Therefore, only a very small fraction of optimal $[\mF,d]_q$-spaces can be obtained using the construction of Remark~\ref{rem:constr}
in the case where $(\mF,d)$ is MDS-constructible and $q$ is large.

\section{Counting MDS-Constructible Pairs}
\label{sec:5}
Since MDS-constructible pairs play a crucial role in the theory of $[\mF,d]_q$-spaces, it is natural to ask how many there are for a given board size $n \times m$ and whether or not most pairs~$(\mF,d)$ are MDS-constructible.
In this section we will completely answer this question for $d=2$ and for $d=3$ when $n=m$.
In the case $d=2$, the number of MDS-constructible pairs $(\mF,d)$ will be  given by 
\textit{Catalan numbers}; see~\cite{stanley2015catalan}.

\begin{proposition} \label{prop:countfs}
The number of $n \times m$ Ferrers diagrams is $\smash{\binom{m+n-2}{n-1}}$.
\end{proposition}

\begin{proof}
We consider an $(n-1) \times (m-1)$ grid associated to $\mF$ as in Figure~\ref{F-path}.
Then $n \times m$ Ferrers diagrams are in bijection with
the paths on the grid from the top left corner to the bottom right corner. Each such path can be seen as a tuple $\smash{(s_1, \dots, s_{m+n-2})}$, where each $s_i$ corresponds to either a $D$ or to an $R$, where $D$ means ``down'' and $R$ means ``right''. Since there are $n-1$ $D$'s and $m-1$ $R$'s in total, there are $\smash{\binom{m+n-2}{n-1}}$ such sequences.
\end{proof}

\begin{notation}
In the sequel, we denote by $\mP(n,m)$ the set of ``down-and-right'' paths in an $(n-1) \times (m-1)$ grid, as in the proof of Proposition~\ref{prop:countfs}.
More formally, an element of $\mP(n,m)$ is a vector in $\{0,1\}^{m+n-2}$ whose entries sum to $m-1$, but we find the graphical representation more convenient for this paper. 
\end{notation}

We include an example that
visualizes how a path in an $(n-1) \times (m-1)$ grid corresponds to an $n \times m$ Ferrers diagram. 

\begin{example} \label{ex:dyck}
Let $n=4$ and $m=6$. The path (in red) in Figure~\ref{F-path} is the sequence $RRDRDRRD$.
     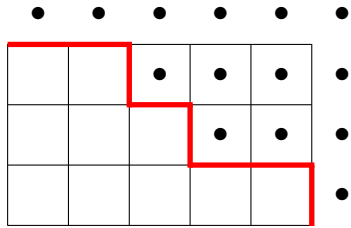
\begin{figure}[ht] 
    \centering
     {\small
     \begin{tikzpicture}[scale=0.8]
     \draw (0,0) -- (5,0) -- (5,3) -- (0,3) -- (0,0);
     \draw (0,1) -- (5,1); 
     \draw (0,2) -- (5,2);
     \draw (1,0) -- (1,3);
     \draw (2,0) -- (2,3);
     \draw (3,0) -- (3,3);
     \draw (4,0) -- (4,3);
     \draw[red, line width=2pt] (0,3) -- (2.043,3);
     \draw[red, line width=2pt] (2,3) -- (2,2-0.043);
     \draw[red, line width=2pt] (2,2) -- (3+0.043,2);
     \draw[red, line width=2pt] (3,2) -- (3,1-0.043);
     \draw[red, line width=2pt] (3,1) -- (5+0.043,1);
     \draw[red, line width=2pt] (5,1) -- (5,0-0.01);
     
     \draw (0.5,3.5) node (b1) [label=center:{\large$\bullet$}] {};
     \draw (1.5,3.5) node (b1) [label=center:{\large$\bullet$}] {};
     \draw (2.5,3.5) node (b1) [label=center:{\large$\bullet$}] {};
     \draw (3.5,3.5) node (b1) [label=center:{\large$\bullet$}] {};
     \draw (4.5,3.5) node (b1) [label=center:{\large$\bullet$}] {};
     \draw (5.5,3.5) node (b1) [label=center:{\large$\bullet$}] {};
     
     \draw (2.5,2.5) node (b1) [label=center:{\large$\bullet$}] {};
     \draw (3.5,2.5) node (b1) [label=center:{\large$\bullet$}] {};
     \draw (4.5,2.5) node (b1) [label=center:{\large$\bullet$}] {};
     \draw (5.5,2.5) node (b1) [label=center:{\large$\bullet$}] {};
     
     \draw (3.5,1.5) node (b1) [label=center:{\large$\bullet$}] {};
     \draw (4.5,1.5) node (b1) [label=center:{\large$\bullet$}] {};
     \draw (5.5,1.5) node (b1) [label=center:{\large$\bullet$}] {};
     
     \draw (5.5,0.5) node (b1) [label=center:{\large$\bullet$}] {};
     
     \end{tikzpicture}
     }
     \caption{Path in $\mP(4,6)$ and its corresponding Ferrers diagram.}
     \label{F-path}
     \end{figure}
The corresponding $4 \times 6$ Ferrers diagram is $\mF=[1,1,2,3,3,4]$.
\end{example}

\begin{remark} \label{rem:wlog}
Define the bijection
\begin{align*}
    T_{n,m}: \{ n \times m \mbox{ Ferrers diagrams $\mF$}\} \rightarrow \{m \times n \mbox{ Ferrers diagrams $\mF$}\},
\end{align*}
where $T_{n,m}(\mF) = \{(m+1-j,n+1-i) : (i,j) \in \mF\} \subseteq [m] \times [n]$ for all $n \times m$ Ferrers diagrams~$\mF$. Note that $T_{n,m}$ is simply the transposition of Ferrers diagrams. 
It is straightforward to check that $T_{n,m}$ induces a bijection between the
MDS-constructible pairs $(\mF,d)$, where 
$\mF$ is an $n \times m$ Ferrers diagram and $1 \le d \le \min\{n,m\}$,
and the MDS-constructible pairs $(\mF,d)$, where 
$\mF$ is an $m \times n$ Ferrers diagram and $1 \le d \le \min\{n,m\}$.
In particular, for all $1 \le d \le \min\{n,m\}$, counting the number of $n \times m$ Ferrers diagrams such that $(\mF,d)$ is MDS-constructible is equivalent to counting the number of $m \times n$ Ferrers diagrams $\mF$ such that $(\mF,d)$ is MDS-constructible. 
This shows that the assumption $m \ge n$, which we include in the statements of this section to simplify the notation and the proofs, is not restrictive.
\end{remark}

We will establish a connection between MDS-constructible pairs and a special class of paths defined as follows.

\begin{definition}
A path $P=s_1, \dots, s_{m+n-2}$ in $\mP(n,m)$ is called a \textbf{generalized Dyck path} if,
for all $\ell \in \{1,\dots,m+n-2\}$,
\begin{align*}
    |\{i \in [\ell] \, : \,  s_i = R \}| \ge |\{i \in [\ell] \, : \,  s_i =D \}|.
\end{align*}
\end{definition}

Note that the path depicted in Figure~\ref{F-path} is a generalized Dyck path. The notion of {Dyck paths} is well-known and it corresponds to the case $n=m$ in our definition of generalized Dyck paths.

We start by analysing the case where $d=2$. Note that for any $n \times m$ Ferrers diagram~$\mF$ with~$m \ge n$, the value of $\kappa(\mF,2)$ can always be obtained by counting the number of dots left in~$\mF$ after deleting the topmost row. In particular, if $|D_i \cap \mF| \ge 1$ for some $i> m$, then $$\sum_{i=1}^{m+n-1} \max \{0,|D_i \cap \mF|-1\} > |\mF|-m = \kappa(\mF,2),$$ meaning that $(\mF,2)$ is not MDS-constructible. From this observation (and its converse) we obtain the following simple  result.

\begin{proposition} \label{lem:charmds2}
Let $\mF$ be an $n \times m$ Ferrers diagrams with $m \ge n \ge 2$. Then $(\mF,2)$ is MDS-constructible if and only if $|D_i \cap \mF|=0$ for all $i >m$.
\end{proposition}

Note that the Ferrers diagrams described in Proposition~\ref{lem:charmds2} are exactly those diagrams whose corresponding path in $\mP(n,m)$ is a generalized Dyck path. From this characterization we are able to count them.

\begin{theorem} \label{thm:countmds2}
The number of $n \times m$ Ferrers diagrams $\mF$ with $m \ge n \ge 2$ for which the pair $(\mF,2)$ is MDS-constructible is  $$\frac{m-n+1}{m} \,  \binom{m+n-2}{n-1}.$$
In particular, the number of 
$n \times n$ Ferrers diagrams $\mF$ for which $(\mF,2)$ is MDS-constructible is the $(n-1)$th Catalan numbers.
\end{theorem}
\begin{proof}
Counting the the number of generalized 
Dyck paths is a classical problem in enumerative combinatorics; see the introduction of~\cite{carlitz1964two} for a closed formula.
\end{proof}

By Theorem~\ref{thm:countmds2} we know that the proportion of $n \times m$ Ferrers diagrams $\mF$ with $m \ge n$ for which the pair $(\mF,2)$ is MDS-constructible within the set of all $n \times m$ Ferrers diagrams is $({m-n+1})/{m}$. In particular, for fixed $n$ and large $m$, almost every Ferrers diagram $\mF$ has the property that~$(\mF,2)$ is MDS-constructible.

In the remainder of this section we concentrate on $n \times n$ Ferrers diagrams (square case).
We start by observing that 
if $\mF$ is an $n \times n$ Ferrers diagram such that $(\mF,2)$ is MDS-constructible,
then Proposition~\ref{lem:charmds2} implies that $\kappa(\mF,3)=\kappa_1(\mF,3)=|\mF|-2n+1$. In particular, it is easy to see that $\sum_{i=1}^{n}\max\{0,|D_i \cap \mF|-2\} = \sum_{i=2}^{n}\max\{0,|D_i \cap \mF|-2\}=|\mF|-2n+1$, by the definition of a Ferrers diagram. This yields the following result.

\begin{corollary} \label{cor:mdsconstrchain}
If $\mF$ is an $n \times n$ Ferrers diagram such that $(\mF,2)$ is MDS-constructible, then also $(\mF,3)$ is MDS-constructible. 
\end{corollary}

\begin{remark} \label{rem:failures}
\begin{enumerate}
    \item If we consider the non-square case ($n \le m-1$), then  Corollary~\ref{cor:mdsconstrchain} is not necessarily true. As a counterexample, consider the $3 \times 4$ Ferrers diagram $\mF=[1,1,3,3]$. Then $\kappa(\mF,2) = \sum_{i=1}^{6} \max\{|D_i \cap \mF|-1,0\}=4$, i.e., the pair $(\mF,2)$ is MDS-constructible. On the other hand, $\kappa(\mF,3) = 2 \ne 1 = \sum_{i=1}^{6} \max\{|D_i \cap \mF|-2,0\}$, meaning that $(\mF,3)$ is not MDS-constructible.
    \item Corollary~\ref{cor:mdsconstrchain} does not extend to larger values of $d$ in an obvious way. Consider e.g. $\mF=[1,1,3,3,5]$. It is easy to see that $(\mF,3)$ is MDS-constructible. However, we have $\kappa(\mF,4)=2 \ne 1 =\sum_{i=1}^9 \max\{0,|D_i \cap \mF|-3\}$, implying that the pair $(\mF,4)$ is not MDS-constructible. 
\end{enumerate}
\end{remark}

We can now count the number of square
Ferrers diagrams $\mF$ for which the pair $(\mF,3)$ is MDS-constructible.

\begin{corollary} \label{cor:mds3}
Let $n \ge 3$. The number of $n \times n$ Ferrers diagrams $\mF$ for which $(\mF,3)$ is MDS-constructible is
\begin{align*}
    \frac{1}{n}\binom{2n-2}{n-1} + \frac{2}{n-1}\binom{2n-4}{n-2}.
\end{align*}
\end{corollary}
\begin{proof}
By Theorem~\ref{thm:countmds2} and  Corollary~\ref{cor:mdsconstrchain}, it suffices to prove that the number of $n \times n$ Ferrers diagrams $\mF$ for which $(\mF,3)$ is MDS-constructible, and $(\mF,2)$ is not MDS-constructible, is
\begin{align*}
    \frac{2}{n-1}\binom{2n-4}{n-2}.
\end{align*}
Fix an arbitrary $n \times n$ Ferrers diagram $\mF$. Let $0 \le \ell \le 2$ be such that
$\kappa(\mF,d) = \kappa_\ell(\mF,d)$, where in the case that there exist more than just one such $\ell$ and one of the possible ones is $\ell=1$, we always set $\ell=1$.
Denote by $\mF_\ell$ the subset of $\mF$ made by those points that are not contained in the topmost $\ell$ rows of $\mF$, nor in its rightmost $d-1-\ell$ columns. By the definition of $\ell$, we have 
\begin{align*}
    \kappa(\mF,3) 
    = \begin{cases}
    |\mF|-2n \quad &\textnormal{if $\ell \in \{0,2\}$},\\
    |\mF|-2n+1 \quad &\textnormal{if $\ell=1$.}
    \end{cases}
\end{align*}
If $\kappa(\mF,3) =  |\mF|-2n+1$, we know that $\kappa(\mF,3)$ can be attained by counting the points of the Ferrers diagram $\mF$ without the topmost row and rightmost column. In particular, in this case $(\mF,3)$ is MDS-constructible if and only if $(\mF,2)$ is MDS-constructible.
Now assume that $\ell =0$. Then $\mF$ has (at least) two columns of length $n$ on the very right. Therefore, $\smash{\sum_{i=1}^{2n-1} \max\{0,|D_i \cap \mF|-2\}=|\mF|-2n}$ if and only if $D_{n+1}\cap \mF=(n,n-1)$ and $|D_i \cap \mF| =0$ for $i > n+1$. Similarly, if $\ell =2$ then $(\mF,3)$ is MDS-constructible if and only if $D_{n+1}\cap \mF=(2,1)$ and $|D_i \cap \mF| =0$ for $i > n+1$. 
\begin{claim} \label{cl:B}
The number of $n \times n$ Ferrers diagrams for which $D_{n+1}\cap \mF=(n,n-1)$ and $|D_i \cap \mF| =0$ for $i > n+1$ is $\smash{\frac{1}{n-1} \binom{2n-4}{n-2}}$.
\end{claim}
\begin{clproof}
These Ferrers diagrams correspond to Dyck paths in $\mP(n-1,n-1)$. The count of these Dyck paths is the formula given in Theorem~\ref{thm:countmds2}, which proves the claim.
\end{clproof}
By symmetry, the formula in Claim~\ref{cl:B} also gives the number of Ferrers diagrams for which \smash{$D_{n+1}\cap \mF=(2,1)$} and $|D_i \cap \mF| =0$ for $i > n+1$, proving the desired result (we are using that $n \ge 3$).
\end{proof}

Finding a formula for the number of 
MDS-constructible pairs $(\mF,d)$ for arbitrary $d$ seems to be a difficult task, where the main hurdle we encountered 
lies in the description of all deletions that attain the minimum in Notation~\ref{not:kappa}.
We leave this to future research.

\bigskip

\bibliographystyle{amsplain}
\bibliography{ourbib}

\providecommand{\bysame}{\leavevmode\hbox to3em{\hrulefill}\thinspace}
\providecommand{\MR}{\relax\ifhmode\unskip\space\fi MR }
\providecommand{\MRhref}[2]{%
  \href{http://www.ams.org/mathscinet-getitem?mr=#1}{#2}
}
\providecommand{\href}[2]{#2}
\begin{thebibliography}{10}

\bibitem{antrobus2019maximal}
J.~Antrobus and H.~Gluesing-Luerssen, \emph{Maximal {F}errers diagram codes:
  {C}onstructions and genericity considerations}, IEEE Transactions on
  Information Theory \textbf{65} (2019), no.~10, 6204--6223.

\bibitem{ballico2015linear}
E.~Ballico, \emph{Linear subspaces of matrices associated to a {F}errers
  diagram and with a prescribed lower bound for their rank}, Linear Algebra and
  its Applications \textbf{483} (2015), 30--39.

\bibitem{braun2016existence}
M.~Braun, T.~Etzion, P.~{\"O}sterg{\aa}rd, A.~Vardy, and A.~Wassermann,
  \emph{Existence of $q$-analogs of {S}teiner systems}, Forum of Mathematics,
  Pi, vol.~4, Cambridge University Press, 2016.

\bibitem{carlitz1964two}
L.~Carlitz and J.~Riordan, \emph{Two element lattice permutation numbers and
  their $q$-generalization}, Duke Mathematical Journal \textbf{31} (1964),
  no.~3, 371--388.

\bibitem{csajbok2017maximum}
B.~Csajb{\'o}k, G.~Marino, O.~Polverino, and F.~Zullo, \emph{Maximum scattered
  linear sets and {MRD} codes}, Journal of Algebraic Combinatorics \textbf{46}
  (2017), no.~3-4, 517--531.

\bibitem{de1981asymptotic}
N.~G. De~Bruijn, \emph{Asymptotic {M}ethods in {A}nalysis}, vol.~4, Courier
  Corporation, 1981.

\bibitem{delsarte1978bilinear}
Ph. Delsarte, \emph{Bilinear forms over a finite field, with applications to
  coding theory}, Journal of Combinatorial Theory, Series A \textbf{25} (1978),
  no.~3, 226--241.

\bibitem{draisma2006small}
J.~Draisma, \emph{Small maximal spaces of non-invertible matrices}, Bulletin of
  the London Mathematical Society \textbf{38} (2006), no.~5, 764--776.

\bibitem{dumas2010subspaces}
J.-G. Dumas, R.~Gow, G.~McGuire, and J.~Sheekey, \emph{Subspaces of matrices
  with special rank properties}, Linear algebra and its applications
  \textbf{433} (2010), no.~1, 191--202.

\bibitem{eisenbud1988vector}
D.~Eisenbud and J.~Harris, \emph{Vector spaces of matrices of low rank},
  Advances in Mathematics \textbf{70} (1988), no.~2, 135--155.

\bibitem{etzion2016optimal}
T.~Etzion, E.~Gorla, A.~Ravagnani, and A.~Wachter-Zeh, \emph{Optimal {F}errers
  diagram rank-metric codes}, IEEE Transactions on Information Theory
  \textbf{62} (2016), no.~4, 1616--1630.

\bibitem{etzion2009error}
T.~Etzion and N.~Silberstein, \emph{Error-correcting codes in projective spaces
  via rank-metric codes and {F}errers diagrams}, IEEE Transactions on
  Information Theory \textbf{55} (2009), no.~7, 2909--2919.

\bibitem{gabidulin}
E.~M. Gabidulin, \emph{Theory of codes with maximum rank distance}, Problemy
  Peredachi Informatsii \textbf{21} (1985), no.~1, 3--16.

\bibitem{GaRe86}
A.~M. Garsia and J.~B. Remmel, \emph{\mbox{$Q$}-counting rook configurations
  and a formula of {F}robenius}, Journal of Combinatorial Theory, Series A
  \textbf{41} (1986), 246--275.

\bibitem{gelbord2002spaces}
B.~Gelbord and R.~Meshulam, \emph{Spaces of singular matrices and matroid
  parity}, European Journal of Combinatorics \textbf{23} (2002), no.~4,
  389--397.

\bibitem{gluesing2020partitions}
H.~Gluesing-Luerssen and A.~Ravagnani, \emph{Partitions of matrix spaces with
  an application to q-rook polynomials}, European Journal of Combinatorics
  \textbf{89} (2020), 103120.

\bibitem{gorla2018rankq}
E.~Gorla, R.~Jurrius, H.~H. L{\'o}pez, and A.~Ravagnani, \emph{Rank-metric
  codes and $q$-polymatroids}, Journal of Algebraic Combinatorics \textbf{52}
  (2020), no.~1, 1--19.

\bibitem{gorla2017subspace}
E.~Gorla and A.~Ravagnani, \emph{Subspace codes from {F}errers diagrams},
  Journal of Algebra and Its Applications \textbf{16} (2017), no.~07, 1750131.

\bibitem{gruica2020common}
A.~Gruica and A.~Ravagnani, \emph{Common complements of linear subspaces and
  the sparseness of {MRD} codes}, SIAM Journal on Applied Algebra and Geometry
  \textbf{6} (2022), no.~2, 79--110.

\bibitem{haglund}
J.~Haglund, \emph{$q$-rook polynomials and matrices over finite fields},
  Advances in Applied Mathematics \textbf{20} (1998), no.~4, 450--487.

\bibitem{koetter2008coding}
R.~Koetter and F.~R. Kschischang, \emph{Coding for errors and erasures in
  random network coding}, IEEE Transactions on Information theory \textbf{54}
  (2008), no.~8, 3579--3591.

\bibitem{lewis2020rook}
J.~B. Lewis and A.~H. Morales, \emph{Rook theory of the finite general linear
  group}, Experimental Mathematics \textbf{29} (2020), no.~3, 328--346.

\bibitem{liu2019constructions}
S.~Liu, Y.~Chang, and T.~Feng, \emph{Constructions for optimal {F}errers
  diagram rank-metric codes}, IEEE Transactions on Information Theory
  \textbf{65} (2019), no.~7, 4115--4130.

\bibitem{lovasz1989singular}
L.~Lov{\'a}sz, \emph{Singular spaces of matrices and their application in
  combinatorics}, Bulletin of the Brazilian Mathematical Society \textbf{20}
  (1989), no.~1, 87--99.

\bibitem{macwilliams1977theory}
J.~MacWilliams and N.~Sloane, \emph{The {T}heory of {E}rror-{C}orrecting
  {C}odes}, Elsevier, 1977.

\bibitem{meshulam1985maximal}
R.~Meshulam, \emph{On the maximal rank in a subspace of matrices}, The
  Quarterly Journal of Mathematics \textbf{36} (1985), no.~2, 225--229.

\bibitem{roth1991maximum}
R.~M. Roth, \emph{Maximum-rank array codes and their application to crisscross
  error correction}, IEEE Transactions on Information Theory \textbf{37}
  (1991), no.~2, 328--336.

\bibitem{schmidt2020quadratic}
K.-U. Schmidt, \emph{Quadratic and symmetric bilinear forms over finite fields
  and their association schemes}, Algebraic Combinatorics \textbf{3} (2020),
  no.~1, 161--189.

\bibitem{segre1955curve}
B.~Segre, \emph{Curve razionali normali e $k$-archi negli spazi finiti}, Annali
  di Matematica Pura ed Applicata \textbf{39} (1955), no.~1, 357--379.

\bibitem{seguins2015classification}
C.~de Seguins~Pazzis, \emph{The classification of large spaces of matrices with
  bounded rank}, Israel Journal of Mathematics \textbf{208} (2015), no.~1,
  219--259.

\bibitem{sheekey2020new}
J.~Sheekey, \emph{New semifields and new {MRD} codes from skew polynomial
  rings}, Journal of the London Mathematical Society \textbf{101} (2020),
  no.~1, 432--456.

\bibitem{silberstein2013new}
N.~Silberstein and A.-L. Trautmann, \emph{New lower bounds for constant
  dimension codes}, 2013 IEEE International Symposium on Information Theory,
  2013, pp.~514--518.

\bibitem{silberstein2015subspace}
\bysame, \emph{Subspace codes based on graph matchings, {F}errers diagrams, and
  pending blocks}, IEEE Transactions on Information Theory \textbf{61} (2015),
  no.~7, 3937--3953.

\bibitem{SKK}
D.~Silva, F.~R. Kschischang, and R.~K\"otter, \emph{A rank-metric approach to
  error control in random network coding}, IEEE Transactions on Information
  Theory \textbf{54} (2008), no.~9, 3951--3967.

\bibitem{stanley2011enumerative}
R.~P. Stanley, \emph{{E}numerative {C}ombinatorics}, 2nd ed., vol.~1, Cambridge
  University Press, 2011.

\bibitem{stanley2015catalan}
\bysame, \emph{Catalan numbers}, Cambridge University Press, 2015.

\bibitem{trautmann2011new}
A.-L. Trautmann and J.~Rosenthal, \emph{New improvements on the
  echelon-{F}errers construction}, 19th International Symposium on Mathematical
  Theory of Networks and Systems, 2010.

\bibitem{zhang2019constructions}
T.~Zhang and G.~Ge, \emph{Constructions of optimal {F}errers diagram rank
  metric codes}, Designs, Codes and Cryptography \textbf{87} (2019), no.~1,
  107--121.

\end{thebibliography}

\end{document}